%% file: Calpha_1pt_v2.tex
\documentclass{amsart}
  \usepackage[top=1in, bottom=1in, left=1.375in, right=1.375in]{geometry}
 \usepackage{amsmath,amssymb}
 \usepackage{algpseudocode}
 \usepackage{algorithm}
 \usepackage{algorithmicx}
 \usepackage{caption}
 \usepackage{subcaption}
 \usepackage{amsthm}
 \numberwithin{equation}{section}
 \usepackage{amsfonts}
 \usepackage{graphicx}
 \usepackage{hyperref}
 \usepackage{makecell}
 \usepackage{longtable}

\input{mymacro}

 \author{Jiajie Chen}
 \date{ \today}

 \address{Courant Institute, NYU, New York, NY 10012. Emails: jiajie.chen@cims.nyu.edu}

\title[Smoothness of $C^{1,\alpha}$ singularity]{Remarks on the smoothness of the $C^{1,\alpha}$ asymptotically self-similar singularity in the 3D Euler and 2D Boussinesq equations}

\begin{document}
\begin{abstract}
We show that the constructions of $C^{1,\alpha}$ asymptotically self-similar singularities for the 3D Euler equations by Elgindi, and for the 3D Euler equations with large swirl and 2D Boussinesq equations with boundary by Chen-Hou can be extended to construct singularity with velocity $\mathbf{u} \in C^{1,\alpha}$ that is not smooth at only one point. The proof is based on a carefully designed small initial perturbation to the blowup profile, and a BKM-type continuation criterion for the one-point nonsmoothness. We establish the criterion using weighted H\"older estimates with weights vanishing near the singular point. Our results are inspired by the recent work of Cordoba, Martinez-Zoroa and Zheng that it is possible to construct a $C^{1,\alpha}$ singularity for the 3D axisymmetric Euler equations without swirl and with velocity $\mathbf{u} \in C^{\infty}(\mathbb{R}^3 \backslash \{0\})$. 
\end{abstract}

 \maketitle

\section{Introduction}

Whether the 3D incompressible Euler equations can develop a finite time singularity from smooth initial data with finite energy is one of the major open questions in nonlinear partial differential equations and fluid dynamics \cite{constantin2007euler,hou2009blow,majda2002vorticity}. There are two major difficulties in the analysis of 3D Euler 
\begin{equation}\label{euler}
   \om_{t} + \uu \cdot \nabla \om = \om \cdot \nabla \uu,  \quad \uu = \na \times (-\D)^{-1} \om,
\end{equation}
where $\uu$ is the velocity and $\om = \na \times \uu$ is the vorticity vector. The first difficulty is the nonlocal Biot-Savart law $ \uu = \na \times (-\D)^{-1} \om$ due to the incompressible conditions $\na \cdot \uu = 0$. The second difficulty is the competing effect between advection and vortex stretching  \cite{lei2009stabilizing,hou2008dynamic,chen2021regularity}. 

In recent years, there has been substantial progress on singularity formation of the 3D Euler equations. In the remarkable work of Elgindi \cite{elgindi2019finite} (see also \cite{elgindi2019stability}), he established singularity formation of 3D axisymmetric Euler equations 
with $C^{1,\al}$ velocity and without swirl. 
In a joint work with Hou \cite{chen2019finite2}, using several important methods developed in \cite{elgindi2019finite}, we establish finite time blowup of the 2D Boussinesq and the 3D axisymmetric Euler equations with $C^{1,\alpha}$ velocity, large swirl and boundary. 
In these works, the regularity exponent $\al$ plays a crucial role as a small parameter, which is used to weaken the advection and simplified the equations and the nonlocal Biot-Savart law $\uu = \na \times (-\D)^{-1} \om$ significantly. 
Building on the framework developed 
in \cite{chen2019finite,chen2021HL},
in joint works with Hou \cite{ChenHou2023a,ChenHou2023b}, we established 
stable nearly self-similar blowup of 2D Boussinesq equations and 3D axisymmetric Euler equations with boundary from smooth, finite energy initial data in finite time. The results provide a first rigorous proof of the Hou-Luo scenario \cite{luo2014potentially,luo2013potentially-2}. The proof combines self-similar analysis, weighted energy estimates, sharp functional inequalities based on optimal transport, operator splitting, and rigorous numerics with computer assistance. One of the key roles of computer assistance is to construct approximate solutions with rigorous error control, 
which provides critical small parameters for the blowup analysis.

In \cite{cordoba2023finite}, Cordoba, Martinez-Zoroa and Zheng discovered a new mechanism 
to establish finite time blowup of 3D axisymmetric Euler equations with velocity $\uu \in C^{1,\al} \cap C^{\infty}(\R^3 \bsh \{ 0\}) \cap L^2$, small $\al$, and without swirl. The construction is based on infinite countable many disjoint smooth vorticity profiles with amplitude solving ODEs of Riccati type. 
Different profiles only concentrate near the origin, which allows the authors to construct smooth solution away from the origin before the blowup. 
There has been some recent progress on constructing $C^{1,\alpha} \cap C^{\infty}(\R^2 \bsh \{0 \})  $ singularity of 2D Boussinesq equations without the boundary. See the talk by Elgindi \footnote{\url{https://www.youtube.com/watch?v=jEHbhxd1XVY}}. 
There are other results on singularity formation in incompressible fluids  \cite{elgindi2018finite,elgindi2017finite,kryz2016} and blowup of model problems for 3D Euler, e.g. \cite{chen2019finite,chen2021HL,chen2020slightly,chen2021regularity,choi2014on}. We refer to the excellent surveys \cite{kiselev2018,Elg22} and the introduction in the author's PhD thesis \cite{chen2022singularity}. 








Inspired by \cite{cordoba2023finite}, we revisit the construction of $C^{1,\al}$ 
asymptotically self-similar singularity in \cite{elgindi2019finite,chen2019finite2}. The velocity $\uu$ in \cite{elgindi2019finite,chen2019finite2} is smooth away from the symmetry axis but is only $C^{1,\alpha}$ near the axis.
Our motivations are the following. Firstly, self-similar analysis has been used successfully to construct blowup for many important PDEs, including incompressible fluids in the above mentioned works \cite{elgindi2019finite,chen2019finite2,ChenHou2023a,ChenHou2023b} and compressible fluids \cite{merle2022implosion1,merle2022implosion2,buckmaster2022formation}. In this work, we show that the self-similar analysis can be used to construct blowup with $\uu$ in the same regularity class as that in \cite{cordoba2023finite}.
In particular, we study the $C^{1,\al}$ (asymptotically) self-similar singularity for the 2D Boussinesq and 3D axisymmetric Euler equations with nonsmoothness at only one point. 
Secondly, it is believed in \cite{elgindi2019finite} that the nonsmoothness in the angular variable, in particular near the axis, is essential to weaken the advection in the construction  of blowup. We show that it is only essential at the singular point for a self-similar blowup.
Note that for the $C^{1,\al}$ singularity \cite{chen2019finite2}, the nonsmoothness can be completely removed, as proved in \cite{ChenHou2023a,ChenHou2023b}. 
For the Hou-Luo scenario considered in \cite{chen2019finite2,ChenHou2023a,ChenHou2023b}, we provide a pen-paper proof of blowup with one-point nonsmoothness by taking advantage of small $\al$.



\subsection{Main results}\label{sec:result}

Recall the 2D Boussinesq equations in $\R^2_+$ 
\beq\label{eq:bous}
\om_t + \uu \cdot \na \om = \th_x, \quad 
\th_t + \uu \cdot \na \th = 0, \quad \uu = \na^{\perp}(-\D)^{-1} \om ,
\eeq
where we impose the no-flow boundary conditions $v(0, y) = 0$.

The axisymmetric 3D Euler equations read
\beq\label{eq:euler1}
\pa_t (ru^{\th}) + u^r (r u^{\th})_r + u^z (r u^{\th})_z = 0, \quad 
\pa_t \f{\om^{\th}}{r} + u^r ( \f{\om^{\th}}{r} )_r + u^z ( \f{\om^{\th}}{r})_z = \f{1}{r^4} \pa_z( (r u^{\th})^2 ).
\eeq
The radial and axial components of the velocity can be recovered from the Biot-Savart law
\beq\label{eq:euler2}
-(\pa_{rr} + \f{1}{r} \pa_{r} +\pa_{zz}) \td{\psi} + \f{1}{r^2} \td{\psi} = \om^{\th}, 
 \quad  u^r = -\td{\psi}_z, \quad u^z = \td{\psi}_r + \f{1}{r} \td{\psi}  .
\eeq
For 3D Euler in a periodic cylinder $D = \{ r,z \in [0, 1] \times \BT \}, \BT = \R / (2 \BZ)$, we impose a no-flow boundary condition on the boundary $r = 1$: $\td{\psi}(1, z ) = 0$ and a periodic boundary condition in $z$.

Our main results are the following. The first result is based on 
\cite{elgindi2019finite,elgindi2019stability}. 
\begin{thm}\label{thm:euler_3D}

There exists $\al \in (0, 1)$ and initial data $\om_0^{\th} \in C_c^{\al}  \cap C^{\infty}(\R^3 \bsh \{0 \})$ without swirl, such that the unique local solution to the axisymmetric 3D Euler equations \eqref{eq:euler_axi} develops an asymptotically self-similar singularity in finite time $T < +\infty$. Moreover, for $t < T$, the solution has regularity $\uu(t) \in C^{1,\al}  \cap C^{\infty}(\R^3 \bsh \{0 \}) \cap L^2$.

\end{thm}

The following two results build on the blowup results with boundary \cite{chen2019finite2}.

\begin{thm}\label{thm:bous}

There exists $\al \in (0, 1)$ and initial data 
$\om_0 \in C_c^{\al}  \cap C^{\infty}(\R^2_+ \bsh \{0 \}), \th_0 \in C_c^{1,\al}  \cap C^{\infty}(\R^2_+ \bsh \{0 \})$, such that the unique local solution $(\om, \th)$ to the  Boussinesq equations \eqref{eq:bous} in $\R^2_+$ develops a focusing asymptotically self-similar singularity in finite time $T < +\infty$. Moreover, for $t < T$, the solution has regularity 
$ 
\uu(t), \th(t) \in C^{1,\al}  \cap C^{\infty}(\R^2_+ \bsh \{0 \}) \cap L^2$.

\end{thm}

\begin{thm}\label{thm:euler_bd}
Consider the 3D axisymmetric Euler equations \eqref{eq:euler1}-\eqref{eq:euler2} in the cylinder $D = \{ (r,z) \in [0, 1] \times \BT \}$ and $O = \{ (r, z) = (1,0)\}$.
 Let $\om^{\th}$ be the angular vorticity and $u^{\th}$ be the angular velocity.
There exists $\al \in (0, 1)$ and some initial data $\om_0^{\th} \in C^{\al}(D) \cap C^{\infty}(D \bsh O) , u_0^{\th} \in C^{1, \al}(D) \cap C^{\infty}(D \bsh O), u_0^{\th} \geq 0$ supported in 
$S = \{ (r, z): |(r, z) - (1,0)| < \f{1}{2} \}$ away from the axis $r=0$, such that the unique local solution to \eqref{eq:euler1}-\eqref{eq:euler2} develops a 
singularity in finite time $T$. 
Moreover, for $t < T$, the velocity field in each period has finite energy with regularity 
$\uu \in C^{1,\al} \cap C^{\infty}( S \bsh O) $ and 
$\om^{\th}, u^{\th} \in  C^{\infty}( D  \bsh O)$. 

\end{thm}

In the space of $(r, z) \in [0,1] \times [-\pi, \pi]$, the set $O$ where the velocity is not smooth is just one point. Yet, due to the axisymmetry, in $\R^3$, the set $O$ is a ring rather than just one point. In Theorem \ref{thm:euler_bd}, we only show that the velocity $u^r, u^z$ is smooth in $S \bsh O$ rather than the larger domain $D \bsh O$ since the velocity is only effective in $\supp(\om^{\th}), \supp( u^{\th}) \subset S$ in \eqref{eq:euler1}. Showing 
$\uu$ is also smooth in $D \bsh S$ is most likely a minor technical issue.

In \cite{elgindi2019finite,chen2019finite2}, the low regularity near the symmetry axis are used to weaken the advection term \eqref{euler} and to construct the approximate blowup profile. 
Theorems \ref{thm:euler_3D}-\ref{thm:euler_bd} show that the low regularity in \cite{elgindi2019finite,chen2019finite2} near the symmetry axis and away from the origin is not essential. Intuitively, since the blowup only occurs near the origin (near $O$ in Theorem \ref{thm:euler_bd}), the regularity of the solution away from the origin is less essential.

Another contribution of  our work is a BKM-type continuation criterion for $C^{\infty}(D\bsh O)$ regularity in various settings, which is used to prove Theorems \ref{thm:euler_3D}-\ref{thm:euler_bd} and is interesting by itself. We quantify the order of singularity of the solution $\na^k \uu$ near the singular point $O$. See Propositions \ref{prop:BKM_bous}-\ref{prop:euler_bd}. In \cite{elgindi2019finite,chen2019finite2}, higher order weighted Sobolev stability estimates $\cH^k$ for the blowup solution are established. However, one needs to first choose the regularity exponent $k$ and then construct initial perturbation. See Theorem 2 \cite{elgindi2019stability} or Theorem \ref{thm19:euler3D}. Given a nonzero initial perturbation, the $\cH^k$ stability estimates do not apply for $k$ large enough. 
The continuation criterion for initial data in Theorems \ref{thm:euler_3D}-\ref{thm:euler_bd} are not trivial since the solutions are quite singular near $x=0$. We use weighted H\"older estimates different from \cite{elgindi2019finite,chen2019finite2} with weights vanishing fast enough near $x=0$. 





\vspace{0.1in}
\paragraph{\bf{Ideas of the proof}} 


In the modified polar coordinate $(R, \b), R = \rho^{\al}$, one of the approximate profiles for the vorticity $\om$ has the form $\bar \Om= C \f{R}{ (1 + R)^2} (\cos \b)^{2 \al/3} (\sin \b)^{\al /3} $. 
Near $R=0$, the perturbation $\td \Om$ needs to be in some weighted space, which implies  $ | \td \Om| \les R^{3/2 + \e}  ( \sin \b \cos \b)^{\g} $ near $R = 0$ for $\g = c \al$ with $ 0< c < 1/3$, and $\e > 0$. If we choose $\td \Om$ in a simple form of $A(R) B(\b)$ near $R=0$, which is used in \cite{chen2019finite2}, \cite{elgindi2019stability},
we get $A(R) \les R^{3/2+}$ and $\bar \Om + \td \Om$ is not smooth near $\b =0,  \f{\pi}{2}$ for small $R$. 

Let $\Phi_i$ be a partition of unity with $R \asymp 2^{-i}$ for $R\in \supp(\Phi_i), i \geq 1$ and $\chi \in C_c^{\infty}$ be supported in $s \leq 2$ with $\chi(s)=1, s  \leq 1$. Instead, we choose the angular perturbation depending on $R$
\[
 \td \Om(R, \b)  \teq    \sum_{i \geq 0}  \Phi_i( R) \B(1 - \chi( \f{\b}{\lam_i \e} ) - \chi( \f{\pi/2 - \b}{\lam_i \e}) \B) \bar \Om - \bar \Om
 =  - \sum_{i \geq 0}  \Phi_i( R) \B( \chi( \f{\b}{\lam_i \e} ) + \chi( \f{\pi/2 - \b}{\lam_i \e}) \B) \bar \Om ,
\]
Perturbation $\td \Om$ removes the nonsmooth part of $\bar \Om$ near $\b=0,  \f{\pi}{2}$. 
The radial part $\bar \Om$ in $\td \Om$ only has a vanishing order $O(R)$ near $R =0$. Yet, using the vanishing term $(\cos \b)^{2 \al/3} (\sin \b)^{\al /3}$ of $\bar \Om$ near $\b=0,  \f{\pi}{2}$, restricting $\td \Om$ to $ \b$ or $\f{\pi}{2} -\b$ very small, 
and choosing $ \lam_i^{\al} \asymp R^k$, we can obtain higher vanishing order in $R$ near $R=0$ and show that $\td \Om$ is small in the energy class. 

We remark that a construction of singularity based on summation over infinite many terms has been used in \cite{cordoba2023finite}. Here, we mainly use the summation for partition of unity and imposing constraints on $\b$ for $(R,\b)\in \supp (\td \Om)$.

For each derivative , e.g. $\pa 
= \f{1}{\rho}  ( \al \cos \b D_R - \sin \b \pa_{\b}) $ \eqref{eq:polar0}, of the initial data $\bar \Om + \td \Om$, near $\rho = |x|=0$, we get a large factor $\rho^{-1} (\lam_i \e)^{-1} \asymp \rho^{-1} \e^{-1} \rho^{-k}$ due to $\lam_i^{\al} \asymp R^k = \rho^{\al k}$. To compensate this large factor, we multiply $\na (\bar \Om + \td \Om)$ by the weight $|x|^{\s} = \rho^{\s}$ near $x=0$ for $\s$ large enough. See the weighted H\"older norm $X_{\s}^{n,\al}$ in \eqref{norm:Xk}. 
We will show in Propositions \ref{prop:BKM_bous}-\ref{prop:euler_bd} that for various settings with a domain $D$ and a point $O$, the $ \bigcap_{n\geq 0} X_{\s}^{n,\al} \subset C^{\infty}(D \bsh O)$ regularity can be continued up to the blowup time using high order weighted H\"older estimates.

\vspace{0.1in}
\paragraph{\bf{Organization of the paper}} The rest of the paper is organized as follows. In Section \ref{sec:pertb}, for various settings with a domain $D$ and a point $O$, we construct  initial data with $C^{\infty}(D \bsh O)$ regularity in Theorems \ref{thm:euler_3D}-\ref{thm:euler_bd}. which lead to a finite time singularity. In Section \ref{sec:BKM}, we study the continuation criterion for the $C^{\infty}(D \bsh O)$ regularity.


\section{Construction of the perturbation}\label{sec:pertb}


In this section, we construct initial perturbation to the (approximate) blowup profiles 
in  \cite{chen2019finite2,elgindi2019finite}.

\vspace{0.1in}
\paragraph{\bf{Notations}}
Recall the following notations and weights from \cite{chen2019finite2,elgindi2019finite}
\[
\bal
& 
D_R = R \pa_R , \quad  D_{\b} = \sin(2\b) \pa_{\b}, \quad 
 || f ||_2 = \B( \int f^2 d R d \b \B)^{1/2}.
 \eal
 \]
The main terms for the velocity $\uu$ in $2D$ and $3D$  has the following form
\beq\label{eq:L12}
\bal
  & L_{2D, 12}( f)(z) = \int_z^{\infty}  \int_0^{\pi/2} \f{ f(R,\b) \sin(2\b)}{R} d R d \b ,  \\ 
& L_{3D,12}(f)(z)  = 3 \int_{z}^{\infty} \int_0^{\pi/2} \f{f(R, \b) \sin(\b) \cos^2(\b)}{R} d \b d R.
\eal
 \eeq

Recall the weights from Section 5.3 \cite{chen2019finite2}:
\beq\label{wg}
\bal
\vp_1 & \teq \f{(1+R)^4}{R^4}  \sin(2\b)^{ - \s}, \quad 
\vp_2 \teq \f{(1+R)^4}{R^4}  \sin(2\b)^{ - \g} , \\
\psi_1  & \teq \f{(1+R)^4}{R^4} (  \sin(\b)\cos(\b) )^{-\s},  \quad
\psi_2    \teq \f{(1+R)^4}{R^4}  \sin(\b)^{-\s}  \cos(\b)^{-\g} , \\
\eal
\eeq
where $\s = \f{99}{100},  \g = 1 + \f{\al}{10}$, and the weighted $H^k$ and $C^1$ norms from Section 6.5 
\cite{chen2019finite2} 
\beq\label{norm:H22}
\bal
 || f ||_{\cH^m(\rho)}  & \teq \sum_{ 0\leq k \leq m}   || D^k_R f \rho_1^{1/2} ||_{L^2} 
+ \sum_{   \ i+ j \leq m - 1
} || D^{i}_R D^{j+1}_{\b}  f \rho_2^{1/2}||_{L^2} , \\
 || f ||_{\cC^1} & = || f ||_{\infty} + || \f{1+R}{R} D_R f ||_{\infty}  + || ( 1 + (R \sin(2\b)^{\al})^{-\f{1}{40}} ) D_{\b}f ||_{\infty} .
\eal
\eeq
We simplify $\cH^3(\vp)$ as $\cH^3$. 

Let $\chi \in C_c^{\infty}(\R), \chi:  \R \to [0, 1]$
with $\chi(s) = 1$ for $s \leq 1$ and 
$\chi(s) = 0$ for $s \geq 2$. Denote 
\beq\label{eq:unity}
\chi_{\lam}(s) \teq \chi( s / \lam) , \quad 
\Phi_0 \teq 1 - \chi, \quad \Phi_i(s) \teq 
\chi( 2^{i-1} s) - \chi(2^{i} s),  i \geq 1 .
\eeq

Clearly, $\{ \Phi_0\}$ forms a partition of unity $\sum \Phi_i(R) = 1$ for $R > 0$. 

We have the following embedding for $\cH^k, \cC^1$ and the standard H\"older space from Lemma 
7.11 in the arXiv version of \cite{chen2019finite2} and Lemma A.13 \cite{chen2019finite2}, 
\begin{lem}\label{lem:h_embed}
(a) For $f \in \cH^3$, we have $|| f ||_{\cC^1} \les \al^{-1/2} || f ||_{\cH^3}$. 
(b) 
For $f(x, y)$ odd or even in $x$. 
we have $|| f||_{C^{  \f{\al}{ 40} }(\R^2_{+}) } \les_{\al} || f ||_{\cC^1(\R^2_{++})}$. 
\end{lem}

In Lemma A.13 \cite{chen2019finite2}, it is proved that $|| f||_{C^{  \f{\al}{ 40} }(\R^2_{++}) } \les || f ||_{\cC^1(\R^2_{++})} $. If $f$ is odd or even in $x$, applying triangle inequality yields  $|| f||_{C^{  \f{\al}{ 40} }(\R^2_+) } \les || f||_{C^{  \f{\al}{ 40} }(\R^2_{++}) }  \les_{\al} || f ||_{\cC^1(\R^2_{++})}$. Result (b) is only used for $\xi$ in Section \ref{sec:bd}.


\subsection{Axisymmetric Euler equations without swirl in $\R^3$}\label{sec:euler3D}

For axisymmetric Euler equations in $\R^3$ without swirl, we rewrite \eqref{eq:euler1} as follows 
\beq\label{eq:euler_axi}
\pa_t \om^{\th} + u^r \pa_r \om^{\th} + u^z \pa_z \om^{\th} = \f{u^r}{r} \om^{\th}.
\eeq
The construction of perturbation is simpler than that of the Boussinesq equations and we use it to illustrate the ideas. 
In  \cite{elgindi2019finite}, the following self-similar blowup profile for $\om^{\theta}$ is constructed 
\beq\label{eq:profile_3D}
\bal
& F = F_* + \al^2 g , \quad || g ||_{\cH^k} \leq C_k, \quad F_* \teq \f{\G_{3D}}{c} \f{4\al R}{
	(1 + R)^2},  \quad 
	\G(\b) \teq ( \sin(\th) \cos^2(\b) )^{\al / 3},   \\
& \rho  = (r^2 +z^2)^{1/2}, \quad R = \rho^{\al}, \quad \b = \arctan (z/ r).
\eal
\eeq
where $R, \b$ are the modified polar coordinate for $(r, z)$, $c \in [ \f{1}{10}, 10]$ is some normalization constant, $\cH^k$ is defined in \eqref{norm:H22}. 

In this section, we consider $\om^{\th}(r, z)$ odd in $z$. Due to symmetry,  we focus on $\b \in [0, \f{\pi}{2} ]$.
Recall $L_{3D,12}, \cH^k = \cH^k(\vp)$ from \eqref{eq:L12},\eqref{norm:H22}. The following results are proved in \cite{elgindi2019finite} and Theorem 2 \cite{elgindi2019stability}. 

\begin{thm}\label{thm19:euler3D}
For $ k \geq 4$, there exists $\al_0 > 0$ and $\d > 0$ such that for initial perturbation 
$\td F_0(r, z)$ odd in $z$ with $|| \td F_0 ||_{\cH^k} < \d_0 \al^{3/2}$, $L_{3D, 12}( \td F_0)(0) = 0$, and initial data $\om_0^{\th} = \td F_0 + F$ with compact support, the unique local solution $\uu \in L^2 \cap C^{1,\al}$ to \eqref{eq:euler_axi} blows up in finite time. 

\end{thm}

Using Theorem \ref{thm19:euler3D} and Proposition \ref{prop:euler_R3} for the continuation of $
X_{\s}^{\infty} \subset C^{\infty}(\R^3 \bsh \{0\})$ regularity, to prove Theorem \ref{thm:euler_3D}, we only need to construct $\td F_0$ with small $||\td F_0||_{\cH^k}$ and $F + \td F_0 \in X_{\s}^{\infty}(\R^3) \cap C_c^{\al}$ for some $\s \geq 1$. Below, we will choose $\s = 14.$


\begin{proof}[Proof of Theorem \ref{thm:euler_3D}]
Let $\{ \Phi_i \}_{i\geq 0}$ be the partition of unity defined in \eqref{eq:unity}. 

For $\lam_i, \e \in [0, 1]$ and $ M\geq 1$ to be defined, we consider the following perturbation  
\[
\bal
 \td F_0(R, \b ) & \teq 
 \td F_1(R, \b) + \mu \td F_2(R, \b), \quad \td F_2 \teq  \sin(2 \b) \chi( (R-3)^3 ) ,  \\
 \td F_1(R, \b) & \teq  \chi( R / M)  \sum_{i \geq 0}  
 \Phi_i( R) \B(1 - \chi( \f{\b}{\lam_i \e} ) - \chi( \f{\pi/2 - \b}{\lam_i \e}) \B) 
  F_*
 - F, \quad  \mu = - \f{ L_{3D, 12}( \td F_1) }{   L_{3D, 12}( \td F_2) },  \\
 \eal
 \]
The rank-one correction $\mu \td F_2$ is used to enforce $L_{3D, 12}(\td F_0) = 0$, which is required in Theorem \ref{thm19:euler3D}.


\vspace{0.1in}
\textbf{Smoothness.}
We have $\td F_2 \in C_c^{\infty},\supp(\td F_2) \subset \{ R \in [1,5] \} $. In \eqref{eq:Hk_euler_small}, we will show that $\td F_1 \in \cH^k$. Using the embedding in Lemma \ref{lem:h_embed}, we get $\td F_0 +  F \in C^{ \al /40}$. By definition, we get 
\[
\td F_0 + F = \chi( R / M)  \sum_{i \geq 0}  
 \Phi_i( R) \B(1 - \chi( \f{\b}{\lam_i \e} ) - \chi( \f{\pi/2 - \b}{\lam_i \e}) \B) 
  F_* + \mu  \td F_2
  \teq \sum J_i + \mu \td F_2,
\]
which has compact support. Since $R = \rho^{\al}$ is smooth away from $ \rho= 0$, $\G(\b)$ \eqref{eq:profile_3D} is smooth away from $\b = 0$ and $\b = \pi/2$, we have  $J_i  \in C_c^{\infty}(\R^3)$. Next, we estimate the $X_{\s}^{k, 0}$ norm \eqref{norm:Xk} with $\lam_i$ chosen in \eqref{eq:lami}. Since $J_i \in C_c^{\infty}$. We only need to estimate $J_i$ with $i$ large enough.
Suppose that $i \geq 2, R \in \supp J_i$. Then $R \asymp 2^{-i}, \rho \leq 1$. Note that  
\beq\label{eq:polar0}
\pa_r = \cos \b \pa_{\rho} - \f{\sin \b}{\rho} \pa_{\b}
 = \f{1}{\rho} ( \al \cos \b  D_R - \sin \b \pa_{\b}), \quad 
 \pa_z = \f{1}{ \rho} (\al \sin \b D_R + \cos \b \pa_{\b}), 
\eeq
$|D_R^k \Phi_i(R) |\les_k \one_{R \asymp 2^{-i}}$, and $\b$ is $\lam_i \e$ away from $0, \pi/2$ for $(R, \b) \in \supp(J_i)$.
For each derivative $\pa_r, \pa_z$, we get a factor $ \f{1}{\rho \lam_i \e} $. To apply Theorem \ref{prop:euler_R3} in full 3D, we also need to estimate 
$\pa_{x_l} \om$, $\pa_{x_l} (J_i e_{\th}) , l = 1,2$, where 
\[
x_1 = r \cos \th, \quad x_2 = r \sin \th, \quad 
\om = \om^{\th} e_{\th} = \om^{\th}( -\sin \th, \cos \th, 0 ) .
\] 
Note that
\[
\pa_{x_1}= \cos \th \pa_r - r^{-1} \sin \th \pa_{\th},
\quad \pa_{x_2} =  \sin \th \pa_r + r^{-1} \cos \th \pa_{\th}.
\]
For each derivative $\pa_{x_l} ( J_i e_{\th} )$, since $J_i e_{\th}$ is smooth in $\th$, $ r = \rho \cos \b \gtr \rho \lam_i \e$ \eqref{eq:profile_3D}, in addition to $\pa_r J_i, \pa_z J_i$, we get a factor $ |r^{-1} J_i|  \les (\rho \lam_i \e)^{-1} J_i $. Therefore, for 3D derivatives, we yield 
\beq\label{eq:Xk_est1}
 | \na^k J_i| \les_k \rho^{-k} (\lam_i \e)^{-k},
 \quad | \la x \ra_{\s}^k  \na^k J_i| \les \rho^{\s k - k} (\lam_i \e)^{-k}
 = R^{ (\s-1) k /\al}(\lam_i \e)^{-k}
\eeq
where we have used $\la x \ra_{\s} \leq |x|^{\s}$ \eqref{eq:xwg}.

Since $R \asymp 2^{-i}, \lam_i = 2^{-12 i /\al} \asymp_{\al} R^{ 12 / \al}$ \eqref{eq:lami}, choosing $\s = 14$, we get estimates
\beq\label{eq:Xk_est2}
 |\la x \ra_{\s}^k \na^k J_i| 
\les_{\al, k} R^{ ( (\s-1) / \al - 12/ \al ) k } \e^{-k}
\les_{\al, k} R^{  k/\al } \e^{-k} \les_{\al, k} \e^{-k},
\eeq
uniform in $i$. For each point $(R, \b)$ away from $(0, 0)$, there are only finite many non-zero terms $J_i$. We get 
\[
 || \td F_0 + F ||_{X_{\s}^{k, 0}} \les_{k, \al} \e^{-k},
\]
and obtain $\td F_0 + F \in  X_{\s}^{k, 0} $. Since $k$ is arbitrary, we get 
$\td F_0 + F \in X_{\s}^{\infty}(\R^3) \cap C_c^{\al} \subset C^{\infty}(\R^3 \bsh \{0 \})$.


\vspace{0.1in}
\textbf{Smallness.} We show that for $\e$ small enough and $M$ large enough, $|| \td F_1||_{\cH^k}$ is small. Firstly, since $F = F_* + \al^2 g$, we get 
\[
\td F_1(R, \b)=
\B( (\chi( \f{R}{M}) - 1) F_* - \al^2 g  \B)
- \chi(R/M) \sum_{i \geq 0}  
\Phi_i(R)
 ( \chi( \f{\b}{\lam_i \e} ) + \chi( \f{\pi/2 - \b}{\lam_i \e}) )  F_* \teq I + II
\]

Note that $\td F_2 \in C_c^{\infty}$. It has been shown in \cite{elgindi2019stability}, in particular Section 2.6 \cite{elgindi2019stability}, that 
\[
|| I||_{\cH^k} \les M^{-1} + \al^2 , \quad | L_{3D, 12}(\td F_2) | >  0, \quad  || \td F_2 ||_{\cH^k} \les 1 , \quad 
|L_{3D, 12}(f)| \les || f ||_{\cH^1}. 
\]

Using the formula of $F_*$ \eqref{eq:profile_3D} and direct calculations, we yield
\[
\bal
 & |D_R^l D_{\b}^m F_*| \les_{l, m} F_*, \quad  D_R^l \Phi_i(R) \les_l \one_{R \asymp 2^{-i}} + \one_{l=0} \one_{R \geq 1}, \quad |D_R^l \chi(R/ M)| \les_l \one_{R \leq 2M},  \\
& |D_{\b}^m \chi(  \f{\b}{\lam_i \e } )| \les_m \one_{\b \leq 2 \lam_i \e}.
\quad |D_{\b}^m \chi( \f{ \pi/2-\b}{\lam_i \e } )| \les_m  \one_{ \pi/2-\b \leq 2 \lam_i \e}.
\eal
\]

Denote by $II_i$ the $i-th$ term in the summation in $II$. Using the above estimates, we get 
\beq\label{eq:pertb1}
|D_R^l D_{\b}^m II_i | 
\les_{l, m} F_* 
\one_{R \leq 2M} ( \one_{R \asymp 2^{-i}} + \one_{i=0} \one_{R \geq 1})
( \one_{ \b \leq 2 \lam_i \e} +   \one_{ \pi/2-\b \leq 2 \lam_i \e} )
\eeq

From the definition of $F_*$ \eqref{eq:profile_3D}, within the support in the above right hand side, we get 
\[
 \G(\b) \les (\sin 2\b)^{\al / 6}   \b^{\al/6} (\pi/2 -\b)^{\al/6}
 \les (\sin 2\b)^{\al / 6}  (\lam_i \e)^{\al / 6}.
\]

In order for $  II_i \in \cH^k$ \eqref{norm:H22}, due to the singular weight $ \f{(1+R)^2}{R^2}$, we need $|II_i| \les R^{3/2 + \d}$ 
for some $\d >0$ near $0$. Since $F_* = O(R)$ near $R=0$, we use the angular vanishing order to compensate it. In particular, we choose $\lam_i = 2^{- 12 i / \al}$. Then for $R \asymp 2^{-i} $ or $R \geq 1$ and $i=0$, we get 
\beq\label{eq:lami}
\lam_i = 2^{- 12 i / \al}, \quad 
\lam_i^{\al/6} \les \min(R^2, 1).
\eeq
As a result, we yield 
\beq\label{eq:pertb2}
|D_R^l D_{\b}^m II_i |  \les \f{R}{(1+R)^2}  (\sin(2\b))^{\al / 6} \e^{\al / 6} \min(R^2, 1)
\les \min(R^3, R^{-1}) (\sin(2\b))^{\al / 6} \e^{\al / 6} .
\eeq
Note that for each $(R, \b) \in \R_+ \times [0, \pi / 2]$, there are finite many non-zero terms $II_i$. The above bound applies to $II$. Using the definition of \eqref{norm:H22}, we obtain 
\beq\label{eq:pertb3}
|D_R^l D_{\b}^m II |  \les \min(R^3, R^{-1}) (\sin(2\b))^{\al / 6} \e^{\al / 6} ,
\quad  || II||_{\cH^k} \les \al^{-1/2} \e^{\al / 6},
\eeq
where $ \al^{-1}$ comes from the weighted integral in $\b$ with weights $(\sin 2\b)^{-\g}$ \eqref{norm:H22},\eqref{wg} and  $\int_0^{\pi/2} (\sin 2\b)^{-1 + c \al} d \al \les \al^{-1}$ for some absolute $c>0$.

Combining the above estimates, we get 
\beq\label{eq:Hk_euler_small}
|| \td F_1 ||_{\cH^k} \les M^{-1} + \al^{-1} \e^{\al/6} + \al^2, 
\ |\mu| \les || \td F_1||_{\cH^k} ,  \
|| \td F_0 ||_{\cH^k} \les || \td F_1 || _{\cH^k} + |\mu| \les M^{-1} +\al^{-1} \e^{\al/6} + \al^2.
\eeq

Choosing $ \al < \al_0$ small enough and then $\e$ small enough and $M$ large, we get $|| \td F_0||_{\cH^k} \leq \d_0 \al^{3/2}$ for the absolute constant $\d$ in Theorem \ref{thm19:euler3D}. We conclude the proof.
\end{proof}

\subsection{2D Boussinesq and 3D Euler with large swirl and boundary}\label{sec:bd}


In this section, we prove Theorem \ref{thm:bous}, which is based on  \cite{chen2019finite2}.
 Firstly, we recall some notations for the Boussinesq equations from Section 2.2 in \cite{chen2019finite2} 
\beq\label{eq:polar_2D}
\bal
& r = \sqrt{x^2 + y^2}, \quad \b = \arctan(y / x), \quad R = r^{\al} = (x^2 + y^2)^{\al / 2}, \\
& \Om(R, \b, t) = \om(x, y, t), \quad \eta( R, \b, t) = (\th_x) ( x, y, t), \quad \xi( R, \b, t) = (\th_y)(x, y, t) ,\\
& J(f)(x, y) \teq \f{1}{x} \int_0^x f(z, y) d z, 
\quad  S = (z^2  + y^2)^{\al/2} , \quad \tau = \arctan( y /z) .
\eal
\eeq
We only use notations $S, \tau$ inside the operator $J$.

Recall the following profiles from Section 4.3 and Appendix A.5.2 in the arXiv version of \cite{chen2019finite2}
\beq\label{eq:profile_2D}
\bal
& \bar{\Om}(R, \b) =  \f{\al}{c_*} \G(\b) \f{ 3R }{ (1 + R)^2}, \quad \bar{\eta}(R, \b) =  \f{\al}{c_*} \G(\b) \f{ 6 R }{ (1 + R)^3} , \quad  \bar{\xi} = \bar{\th}_y  =  \int_0^x  \bar{\eta}_y(z, y) dz,  \\
& \G(\b) = (\cos(\b))^{\al}, \quad  c_* =  \f{2}{\pi} \int_0^{\pi/2} \G(\b) \sin(2\b) d\b ,  
\quad \bar{\th}   = 1 + \int_0^x \bar{\eta}(z, y) dz = 1 + x J( \bar \eta) , \\
&   
 \bar{\eta}_y(x, y)  = - \f{ 6  \al}{c}  \cdot  \f{3\al y}{y^2 + x^2} 
 \f{  (x^2 +y^2)^{\al /2}  x^{\al }} {(1+(x^2 +y^2)^{\al/2})^4} 
 = - \f{ 18  \al^2}{c} \f{\sin \b}{r} \f{ R^2 (\cos \b)^{\al} }{(1+R)^4}.
 \eal
\eeq
Note that in \cite{chen2019finite2}, $\bar \th = x J(\bar \eta)$. To obtain $\bar \th^{1/2} \in C^{1,\al}$, we use the above construction in the updated arXiv version of \cite{chen2019finite2}. See \cite{chen2023correction}. We recall the following estimates of the profile $\bar \Om, \bar \eta,\bar \xi$ from Lemmas A.6, A.8 \cite{chen2019finite2}, Propositions A.7, Proposition 7.6, the identities for $J(f)$ from Lemma A.12 \cite{chen2019finite2}.
\begin{lem}\label{lem:profile_2D}
For $R \geq 0, \b \in [0, \pi/2], i+j \leq 5$, we have $\bar \Om, \bar \eta, \bar \xi \in C^{\al/40}$,
\[
 | D_R^i D_{\b}^{j} \bar \Om | \les (\al \sin(\b) )^{l(j)} \bar \Om,
\quad 
  | D_R^i D_{\b}^{j+1} \bar \eta | \les (\al \sin(\b) )^{l(j)}  \bar \eta,  \ l(j) = \one_{j\geq 1}, 
   \quad |D_R^i D_{\b}^j \bar \xi| \les - \bar \xi.
\]

Suppose that $f(0, y) =0$. Recall  $\tau, S$ from \eqref{eq:polar_2D}. We have
\beq\label{eq:Jop}
\bal
D_R J(f)(x, y) & = J( D_S f)(x,y) ,  \\ 
D_{\b} J(f)(x, y) - J( D_{\tau} f)(x,y)  & = -2 \al \sin^2 (\b) \cdot J( D_S f) 
+  2 \al J( \sin^2(\tau)  D_S f) , \\
\eal
\eeq
\end{lem}

Recall the following energies from Section 8.6.1 in \cite{chen2019finite2}
\beq\label{eq:EE}
\bal
E^2(\Om, \eta, \xi) &=   L^2_{2D, 12}(\Om)(0) + \int \Om^2 (R^{-3} + 1) \sin(2\b) +  \eta^2 (R^{-4} + R) \G(\b)^{-1}  d R d \b \\
& + || \Om||^2_{\cH^3} + || \eta||_{\cH^3}+ || \xi  ||^2_{\cH^3(\psi)} 
+ \al || \xi ||_{\cC^1}^2 .\\
\eal
\eeq
where $L_{12,2D}$ is defined in \eqref{eq:L12}. The weighted $L^2$ part,  $\cH^3$ part,  and $\cC^1$ part are first introduced in Section 5.6, Section 6.5, and at the beginning of Section 8 in \cite{chen2019finite2}. The above energies are different but equivalent to the final energy in Section 8.6.1 in \cite{chen2019finite2} up to universal constants independent of $\al$. We drop the parameters for different parts of norms in $E$ to simplify the notations. The following stable blowup results are established in \cite{chen2019finite2}.

\begin{thm}
\label{thm19:bous}
There exists $\al_0 >0$ and $K_* >0$, such that for $0 < \al < \al_0$ and initial perturbation
$ \td \om_0(x,y)$ odd in $x$ ,  $\td \th_0(x,y) $ even in $x$ with 
$E(\td \om_0 , \td \th_{x,0}, \td \th_{y,0} ) < K_* \al^2$, the solution from the initial data $\bar \om + \td \om_0 \in C^{\al}, \td \th_0 +\bar \th_0 \in C^{1,\al}$ to the Boussinesq equation \eqref{eq:bous} develop an focusing asymptotically self-similar singularity in finite time. In particular, we can choose $\td \om_0$ such that the initial velocity $\td \uu_0 + \bar \uu \in L^2$. 

\end{thm}

Using Theorem \ref{thm19:bous} and Proposition \ref{prop:BKM_bous} for the continuation of $
X_{\s}^{\infty} \subset C^{\infty}(\R^2_+ \bsh \{0\})$ regularity, we only need to construct perturbation $\td \om, \td \th$ 
small in the energy class, such that  $\bar \om + \td \om_0 \in C_c^{\al} \cap X_{\s}^{\infty}, \td \th_0 +\bar \th_0 \in C_c^{1,\al}, \na( \td \th_0 +\bar \th_0)  \in X_{\s}^{\infty} $. The construction is more technical since we do not have the explicit formula for $\bar \th$. Moreover, $\bar \th $ is not smooth near $ \b =0, \b = \f{\pi}{2}$, if we modify the angular part of $\bar \th$ near $ \b = 0$ directly, it will lead to a large perturbation. We will choose $\s = 28$.


\begin{proof}[Proof of Theorem \ref{thm:bous}]
We first construct $\td \Om_0$ and modify $\bar \eta$. Then we construct $\td \th_0,\td \eta_0, \td \xi_0$.

\vspace{0.1in}
\paragraph{\bf{Step 1: Construction of $\td \Om_0,\hat \eta_0$} }
Recall the partition of unity $\{ \Phi_i\}_{i \geq 0}$ \eqref{eq:unity} and $\lam_i$  \eqref{eq:lami}. For $M\geq 1, \e < 1$ to be defined, we construct the following perturbation to $\bar \Om$ 
and modification to $\bar \eta$
\[
\bal
  \td \Om_0(R, \b) & \teq  \chi( \f{R}{M})  \sum_{i \geq 0}  
 \Phi_i( R) \B(1  - \chi( \f{\pi/2 - \b}{\lam_i \e}) \B)  \bar \Om
- \bar \Om 
= ( \chi( \f{R}{M})  - 1) \bar \Om - \chi( \f{R}{M}) 
\sum_{i \geq 0}  
 \Phi_i( R)  \chi( \f{\pi/2 - \b}{\lam_i \e})   \bar \Om ,  \\
 \hat \eta(R, \b) 
 & \teq  \sum_{i \geq 0}  
 \Phi_i( R) \B(1  - \chi( \f{\pi/2 - \b}{\lam_i \e}) \B)  \bar \eta
 = \bar \eta - 
 \sum_{i \geq 0}   \Phi_i( R)  \chi( \f{\pi/2 - \b}{\lam_i \e})  \bar \eta .
 \eal
 \] 
 We only modify the angular part near $\b = \f{\pi}{2}$ since $\G(\b)$ \eqref{eq:profile_2D} is smooth near $\b = 0$. We will localize functions related to $\hat \eta$ and construct $\td \th$ later.  Following the smoothness part in the proof of Theorem \ref{thm19:euler3D} in Section \ref{sec:euler3D}, we obtain 
 \beq\label{eq:regu1} 
 \bar \Om + \td \Om_0 \in C_c^{\al} \cap X_{\s}^{\infty} \subset C^{\infty}(\R^2_+ \bsh \{0 \}), \quad 
\hat \eta \in C^{\al} \cap X_{\s}^{\infty} \subset C^{\infty}(\R^2_+ \bsh \{0 \}), \quad \s \geq 14.
 \eeq
Since $ \na^k \hat \eta$ has decay $|x|^{-k-\al}$ and $\la x \ra_{\s}\leq |x|$,
 $ \la x \ra_{\s}^k\na^k \hat \eta $ is bounded for $|x| \geq 1$. Following the estimates \eqref{eq:pertb1}-\eqref{eq:pertb3} and using the estimates for $\bar \eta, \bar \Om$ in Lemma \ref{lem:profile_2D} and
\[
\one_{\pi/2 - \b \leq 2 \lam_i \e} \leq \one_{\pi/2 - \b \leq 2 \e} ( \sin \b), \ 
|D_R^i D_{\b}^j \bar \Om| \les \bar \Om  (\al \sin \b)^{l(j)},  \
|D_R^i D_{\b}^j \bar \eta| \les \bar \eta  (\al \sin \b)^{l(j)},   \   
\]
where $l(j) = \one_{j \geq 1 }$, we yield 
 \beq\label{eq:small1}
 \bal
& | D_R^i D_{\b}^j \td \Om_0| \les 
\one_{R \geq M} R^{-1} \G(\b)  (\al \sin \b)^{l(j)}
+\e^{\al/6} \min(R^3, R^{-1}) (\cos \b)^{5 \al/6}   \sin \b , \\
& | D_R^i D_{\b}^j ( \hat \eta - \bar \eta)| \les 
\min(\bar \eta,
\e^{\al/6} \min(R^3, R^{-2}) (\cos \b)^{5 \al/6}  (\al \sin \b)^{l(j)} ) \one_{\pi/2 - \b \leq 2  \e} .
\eal
 \eeq
 The key point is that the vanishing order in $R$ and $\b$ in the above estimates are higher than those required for boundedness of the energy \eqref{eq:EE}, \eqref{norm:H22}. For $\eta$, we have the faster decay rate $R^{-2}$ since $ \bar \eta \les R^{-2}$ \eqref{eq:profile_2D}. Plugging the above estimates into the energy \eqref{eq:euler1} and using a direct calculation, we obtain that the smallness of $\td \Om_0$
 \beq\label{eq:small2D_om}
 |L_{12}(\td \Om_0)| + 
  || \td \Om_0 ||_{\cH^3}  
   \les \al^{-1} ( M^{-1} + \e^{\al / 6} ), \quad 
 E( \td \Om_0, 0, 0) \les 
\al^{-1} ( M^{-1}+ \e^{\al / 6} ),
 \eeq
 where we get $\al^{-1}$ due to the same reason as that below \eqref{eq:pertb3}.

\vspace{0.1in}
\paragraph{\bf{Step 2: Construction of $\hat \th, \td \eta_0, \td \xi_0$ and smoothness}}
In \eqref{eq:profile_2D}, we recover $\bar \th$ using $\bar \th(0, y) = 1$ and integrating 
$\bar \eta$. Since $\bar \eta$ is not smooth near $0$, $\bar \th$ is not smooth even away from $0$. To overcome it, we want to integrate $\bar \eta$ from $x$ to $\infty$ to recover $\bar \th$. Yet, if we do so for all $\bar \th(x, y)$, it will lead to a large perturbation. Instead, we modify the integration only for small $y$.

For small $ \d \in (0, 1/2) $ and $f(y)$ to be determined, we construct 
\beq\label{eq:pertb_profi}
\hat \th \teq f( y) + \int_0^{x} \hat \eta(z, y) d z,  
 \quad f_y(y) \teq  -\chi( \f{y}{\d}) \int_0^{\infty} \hat \eta_y(z, y) d z, 
 \quad f(y) = 1 - \int_y^{\infty} f_y d y,
\eeq
where $\chi \in C_c^{\infty}$ with $\chi(s) = 1, s\leq 1, \chi(s)=0, s \geq 2$. By definition, we yield 
\beq\label{eq:xi_hat}
\hat \th_x = \hat \eta,  \quad 
\hat \th_y = f_y(y) + \int_0^x \hat \eta_y(z, y) d y
= (1 -  \chi( \f{y}{\d}) ) \int_0^{\infty} \hat \eta_y(z, y) d z
 - \int_x^{\infty} \hat \eta_y(z, y) d z \teq \hat \xi_1 + \hat \xi_2.
\eeq

For $M >1$ large, we localize $\hat \th$ and construct the perturbation 
\beq\label{eq:pertb_th}
\td \th_0 \teq \chi(R/ M) \hat \th - \bar \th, \quad 
\td \eta_0 \teq \pa_x ( \chi(R/ M) \hat \th ) - \bar \eta,
\quad \td \xi_0 \teq \pa_y ( \chi(R/ M) \hat \th ) - \bar \xi,
\eeq
Since $\hat \eta \in C^{\infty}(\R^2_+ \bsh \{0\})$ \eqref{eq:regu1}, 
we get $\hat \th_x, \hat \th_y \in C^{\infty}( \R^2_+ \bsh \{0 \} )$, $\one_{R \leq 2 M} \hat \th \in L^{\infty}$. 
We will justify that $\hat \eta_y$ is integrable in \eqref{eq:xi_hat3}. The above construction provides initial data $\td \th_0 + \bar \th \in C_c$.

Next, we show that $( \td \eta_0 + \bar \eta, \td \xi_0 + \bar \xi) = \na ( \chi(R/ M) \hat \th)
\in X_{\s}^{\infty} $
 for some $\s \geq 1$. We have shown $\hat \eta \in X^{\infty}_{\s}, \s \geq 14$ \eqref{eq:regu1}.
For $\hat \th_y = \hat \xi_1 + \hat \xi_2$ defined above, using the estimates 
\[
| \na^k \hat \eta| \les \la x \ra_{\s_0}^{-k}|| \hat \eta ||_{X_{\s_0}^{k,0}} \les 
   (|x|^{- k\s_0}+ |x|^{-k} ) || \hat \eta ||_{X_{\s_0}^{k,0}}, \
\s_0 = 14 ,
\]
we yield
\[
\bal
&  |\pa_1^{i+1} \pa_2^j \hat \xi_2 | 
=  |\pa_1^i \pa_2^{j+1} \hat \eta|  
\les \la x \ra_{\s_0}^{-i-j-1} || \hat \eta||_{X_{\s}^{i+j+1, 0}}, \quad i, j \geq 0, \\
&   | \pa_2^j \hat \xi_2 |  
 \leq \int_{x_1}^{\infty} \la (z, x_2) \ra_{\s}^{-j - 1} d z \cdot || \hat \eta ||_{X_{\s_0}^{j+1,0}} 
 \les \B(  |x| \cdot \la x \ra_{\s}^{- (j+1) }  
 + \int_{|x|}^{\infty} |(z, x_2)|^{-(j+1) \s_0} +|(z, x_2) |^{-j-1} d z \B) || \hat \eta ||_{X_{\s_0}^{j+1,0}} \\ 
 & \qquad \les  ( |x|^{ 1 - (j+1) \s_0 }  + |x|^{-j} ) || \hat \eta ||_{X_{\s_0}^{j+1,0}} 
 \les ( |x|^{- 2 j \s_0} + |x|^{-j } ) || \hat \eta ||_{X_{\s_0}^{j+1,0}} ,  \ j \geq 1  .
 \eal
\]

We will only apply the second estimate for $j \geq 1$. For $\hat \xi_1$ \eqref{eq:xi_hat}, it is supported in $y \geq \d$ and depends on $y$. We yield 
\[
 |\pa_2^j \hat \xi_1| 
 \les_{j, \d} \one_{y \geq \d} \int_0^{\infty}   |(z, y) |^{-(j+1) \s_0}  d z  || \hat \eta ||_{X_{\s_0}^{j+1,0}} 
 \les_{j, \d} \one_{y \geq \d}  y^{-(j+1 ) \s_0 + 1}  || \hat \eta ||_{X_{\s_0}^{j+1,0}} 
 \les_{j, \d} \d^{-(j+1) \s_0 +1} || \hat \eta ||_{X_{\s_0}^{j+1,0}} .
\]

To estimate $\hat \th_y$, we use the first formula in \eqref{eq:xi_hat} and \eqref{eq:xi_hat3} with $k = 0$ and \eqref{eq:pertb_f} to obtain $ |\hat \th_y| \les_{\e, \al} |y|^{2\al} $. Using the localization $\chi(R/ M)$ and the above estimates, we obtain 
\[
| |\la x \ra_{2 \s_0}^{k} \na^k \hat \th_y| \one_{|x| \leq 2M}  =
 |\la x \ra_{2 \s_0}^{k} \na^k (\hat \xi_1 + \hat \xi_2) | \one_{|x| \leq 2M} \les_{M, \d, \e} 
|| \hat \eta ||_{X_{\s_0}^{k+1,0}}  + 1
 , 
 \quad k \geq 0.
\]
Since $\hat \th \one_{|x| \leq 2M}$ is bounded, using Leibniz's rule and \eqref{eq:regu1}, we prove 
\[
\td \th_0 + \bar \th \in C_c^0, \quad 
  \td \eta_0 + \bar \eta,  \ \td \xi_0 + \bar \xi \in C_c \cap X_{2\s_0}^{\infty}  \subset C_c^{\infty}(\R^2_+ \bsh \{0\})  , 
\]

 In step 4, we will show that for $\d$ small enough and $M$ large enough, the perturbations $\td \eta_0, \td \xi_0$ are small in the energy \eqref{eq:EE}. 
 Using embedding in Lemma \ref{lem:h_embed} and Lemma \ref{lem:profile_2D}, we further obtain $\td \eta_0, \td \xi_0, \td \eta_0 + \bar \eta, \td \xi_0 + \bar \xi \in C^{\al / 40} $.




\vspace{0.1in}
\paragraph{\bf{Step 3: Estimate of $f$}} Firstly, we estimate $(y\pa_y)^k \bar \eta_y$
and then $(y\pa_y)^k \pa_y (\bar \eta-\hat \eta)$.


We have the following formula for the derivatives in polar coordinate
\beq\label{eq:DR}
 r \pa_r = \al R \pa_R = \al D_R, \quad \pa_y = \sin(\b) \pa_r + \f{\cos \b}{r} \pa_{\b}
= \sin(\b) \f{\al D_R}{r} + \f{ 1 }{ 2 \sin(\b) r} D_{\b}
\eeq

 Using the formula of $\bar \eta_y$ \eqref{eq:profile_2D}, $ y \pa_y = \sin^2(\b) \cdot \al D_R + \f{1}{2} D_{\b}$ \eqref{eq:DR} and a direct computation, we obtain 
\[
|D^k_{R} R| \les_k R,  \ D_r^k r^{-i} \les_{k, i} r^{-i},   \
|D^k_{\b} q(\b)| \les_k q(\b), \ q = (\cos \b)^{\al}, \sin \b , \cos \b,  \ 
|(y\pa_y)^k  \bar \eta_y| 
\les_k |\bar \eta_y| = - \bar \eta_y, 
\]

Using \eqref{eq:profile_2D} and the above estimates, we get 
\[
\bal
& |\bar \eta_y(z, y)| \les  \f{ \al^2 y}{y^2 + z^2} (z^2 + y^2)^{\al/2} z^{\al},  \quad 
   \int_0^{\infty} |(y\pa_y)^k \bar \eta_y(z, y)|  d z 
 \les_k  \int_0^{\infty}\f{ \al^2 y}{y^2 + z^2} (z^2 + y^2)^{\al/2} z^{\al} d z  .
 \eal
\]
Using a change of variable $ z = y s$, we further yield 
\beq\label{eq:small2}
\int_0^{\infty} |(y\pa_y)^k \bar \eta_y(z, y)|  d z  \les \al^2 y^{2 \al }
\int_0^{\infty} (1+ s^2)^{-1 + \al / 2} s^{\al} d s \les \al^2 y^{2\al}.
\eeq
for $\al \leq \f{1}{10}$. Note that we will choose $\al$ to be very small. 

Using  \eqref{eq:DR} and then \eqref{eq:small1} for $ \bar \eta-\hat \eta$, we get 
\beq\label{eq:Jop_eta1}
|(y \pa_y)^k \pa_y (\bar \eta - \hat \eta) |
\les r^{-1} \sum_{ i+ j \leq k+1} |D_R^i D_{\b}^j (\bar \eta - \hat \eta)|
\les_k r^{-1} \cdot  \e^{\al/6}  \min(R^3, R^{-2}) \one_{\pi/2 -\b \leq 2\e} 
\teq P. 
\eeq

By choosing small $\e$,  if $\tan(\b) = y / x$ with $\pi/2-\b \leq 2\e$, we get $x / y = \tan(\pi/2-\b) \leq 2(\pi/2-\b) \leq 4 \e < \f{1}{2}$. It follows $(x^2 + y^2)^{1/2} \asymp y$, and
\beq\label{eq:Jop_eta2}
P \les \e^{\al/6} \f{1}{y} \cdot
y^{3\al}  \one_{x \leq 4 \e y} , \quad  \int_0^{\infty} |(y\pa_y)^k\pa_y( \bar \eta - \hat \eta)(z, y)| d z
 \les_k \e^{\al/6} \int_0^{4\e y} y^{3\al - 1} d z
 \les_k \e y^{3\al}.
\eeq

Combining \eqref{eq:small2} and the above estimate, for $y \leq 1$, we obtain 
\beq\label{eq:xi_hat3}
\int_0^{\infty} |(y\pa_y)^k\pa_y  \hat \eta(z, y)| d z \leq C_1 (\al^2 + \e) y^{2\al}, 
\quad k \leq 5, 
\eeq
for some universal constant $C_1>0$. Recall the formula of $f, f_y$ \eqref{eq:pertb_profi}. Since 
\[
|(y\pa_y)^k \chi( y/ \d) | \les \one_{y \leq 2\d}, \quad \chi(y/\d) \leq \one_{y \leq 2\d},
\]
we choose $ \d < \min(\f{1}{4}, \f{1}{4C_1}) $ so that 
\beq\label{eq:pertb_f}
\bal
& |(y\pa_y)^k f_y(y)| \les   \one_{ y < 2 \d }  ( \al^2 + \e) y^{2\al}, \  k \leq 5, \quad 
| (y \pa_y)^k (f(y) - 1) | \les |y| \one_{ y < 2 \d }  ( \al^2 + \e) y^{2\al}, 1 \leq k \leq 5, \\
& |f_y(y)| \leq  C_1 \one_{ y < 2 \d }  ( \al^2 + \e) y^{2\al}, \quad |f(y) - 1| \leq (\al^2 + \e) 2 \d C_1 \one_{y \leq 2 \d} < (\al^2 + \e ) \one_{y \leq 2\d}.
\eal
\eeq
For $\al, \e$ small, since $\hat \eta > 0$, from \eqref{eq:pertb_profi}, we get
\[
\hat \th(x, y) \geq  \th(0, y) = f(y) \geq \f{1}{2} > 0.
\]

\vspace{0.1in}
\paragraph{\bf{Step 4: Smallness}}
In Lemma A.11 \cite{chen2019finite2}, the truncation on $\bar \th$ is proved to be small
\beq\label{eq:small2D_th0}
\bal
& \lim_{M \to \infty} E( 0, \td \eta_1, 0 ) + || \td \xi_1 ||_{\cH^3(\psi)} = 0, 
\quad \limsup_{M\to \infty} || \td \eta_1 ||_{\cC^1} \leq C \al^2,  \\
& \td \eta_1 \teq \pa_x( (\chi(R/M) - 1 ) \bar \th ), \quad  \td \xi_1 \teq \pa_y( (\chi(R/M) - 1 ) \bar \th ).
\eal
\eeq

We decompose $\td \eta_0$ \eqref{eq:pertb_th}  as follows 
\[
\td \eta_0 = 
 \pa_x(  (\chi(R/M) - 1) \bar \th)+ \pa_x( \chi(R/ M ) (\hat \th - \bar \th) ) 
\teq \td \eta_1 + \td \eta_2,
\]
and perform similar decomposition for $\td \xi_0$. It remains to estimate $\td \eta_2, \td \xi_2$. For $\td \eta_2, \td \xi_2$, we get 
\beq\label{eq:pertb_decom}
\bal
\td \eta_2 = \chi(R/ M) \pa_x(\hat \th - \bar \th) 
+ \f{1}{r} \cos (\b ) \f{R}{M}  \chi^{\prime}( \f{R}{M}) (\hat \th - \bar \th)
\teq I_1 + I_2 ,\\
\td \xi_2 = \chi(R/ M) \pa_y(\hat \th - \bar \th) 
+ \f{1}{r} \sin (\b ) \f{R}{M}  \chi^{\prime}( \f{R}{M}) (\hat \th - \bar \th)
\teq I_3 + I_4 ,\\
\eal
\eeq
Using \eqref{eq:small1}, $|D_R^k \chi(R/ M)| \les 1$, and a direct computation similar to \eqref{eq:small2D_om}, we get 
 \beq\label{eq:small2D_th1}
 E(0, I_1 , 0) \les \al^{-1} \e^{\al / 6}. 
\eeq

For $I_2$, comparing \eqref{eq:profile_2D} and \eqref{eq:pertb_th}, and using \eqref{eq:pertb_f} 
and the operator $J$ in \eqref{eq:polar_2D}, we yield 
\beq\label{eq:pertb7}
\hat \th- \bar \th 
= f(y) - 1 + x J( \hat \eta - \bar \eta),
\quad \f{1}{r} (\hat \th- \bar \th ) 
= \f{1}{r} (f(y)-1) + \cos(\b)  J( \hat \eta - \bar \eta),
\eeq

Using Lemma \ref{lem:profile_2D}, we can almost commute the derivatives $D_R, D_{\b}$ with $J$. Since 
\[
|D_R^i D_{\b}^j g| \les 1, \quad g = \sin \b , \cos \b ,  \chi(R/ M), 
\]
using \eqref{eq:small1} and \eqref{eq:Jop} for $J(\bar \eta)$,  we yield 
\[
|D_R^i D_{\b}^j J( \hat \eta - \bar \eta) | 
\les_{i,j}  J(  P_2 ) , 
\quad 
P_2 \teq \e^{\al / 6} \min(R^3, R^{-2}) (\cos(\b))^{5\al/6} \one_{\pi/2-\b \leq 2 \e}.
\]

For $\e$ small, following the estimates in \eqref{eq:Jop_eta1}, \eqref{eq:Jop_eta2}, within the support of $P_2$, we get $x / y \leq 4 \e,  x^2 + y^2\asymp y^2, \cos(\b) = x / r \asymp x / y$. We obtain 
\beq\label{eq:Jop_eta3}
\bal
 P_2(z, y) & \les \e^{\al / 6} \min( y^{3\al}, y^{-2\al})  \cdot (z/y)^{5\al/6} \one_{z \leq 4 \e y}
\teq A(y) \cdot (z/y)^{5\al/6} \one_{z \leq 4 \e y}, \\
 J(P_2) & = \f{1}{x}\int_0^x P_2(z, y) d z
\les A(y) \f{1}{x} \int_0^{\min(x, 4\e y)} \f{ z^{5\al / 6} }{y^{5\al/6} } d z  \les A(y)  \f{ \min(4 \e y, x)^{1 + 5\al / 6} }{ x y^{5\al / 6}} \\
& \les \e^{\al / 6} \min( (R \sin(\b))^{3\al}, (R \sin \b)^{-2\al}) 
\min( (\cos \b)^{5\al / 6}, \tan \b  )  \\
& \les \e^{\al / 6} \min(R^3, R^{-2})  ( \cos \b )^{5\al/6} (\sin \b )^{1/2}) ,
\eal
\eeq
where the last inequality can be obtained by discussing $\b \leq \pi/4$ and $\b>\pi/4$.

Since $y = r \sin(\b), D_R = \f{1}{\al} D_r$ \eqref{eq:DR}, \eqref{eq:polar_2D} and
\beq\label{eq:Dy}
\bal
& D_R q(y) = \al^{-1} r \pa_r q(y) = \al^{-1} r \sin(\b) q^{\prime}(y)
 = \al^{-1} y q^{\prime}(y),  \\
 & D_{\b} q(y) = \sin(2\b) \pa_{\b} q(y) =  \sin(2\b) r \cos(\b) q^{\prime }(y)
 = 2 \cos^2(\b) y q^{\prime}(y), 
\eal
\eeq
for $i+ j \leq 5$, using \eqref{eq:pertb_f}, we obtain 
\[
|D_R^i D_{\b}^j ( f(y) - 1 ) | 
\les \al^{-i} \sum_{k\leq i+j} |(y\pa_y)^k (f(y)-1)| 
\les \al^{-5}  (\al^2 + \e) |y| \one_{|y|\leq 1}.
\]

For $I_2, I_4$ \eqref{eq:pertb_decom}, due to the cutoff function $\chi^{\prime}(R/ M)$, within the support, we get $R \asymp M, r = R^{1/\al} \geq 1$. Therefore, for $(q, l) = (\cos(\b), 2)$ or $(\sin(\b), 4)$, combining the above estimates and using \eqref{eq:pertb7}, for $i+j \leq 5$, we yield 
\[
|D_R^i D_{\b}^j  I_l|
\les \al^{-5} r^{-1}( \al^2 + \e ) \one_{R \geq M} 
|y| \one_{|y| \leq 1} + \e^{\al / 6}\min(R^3, R^{-2}) ( \cos \b)^{5\al/6} ( \sin \b)^{1/2} .
\]

For the first term, since $y = r \sin(\b) \leq 1,  r = R^{1/\al} \geq M^{1/\al}$, within the support, we get 
$(\sin(\b))^{\al} \leq R^{-1} \leq M^{-1}, \sin \b \leq r^{-1} $, and 
\[
r^{-1} |y| = \sin \b \les r^{-1/2} \sin(\b)^{1/4} \sin^{1/4}(\b) \cos(\b)
\les R^{-1/(2\al)} M^{-1/(4\al)}\sin^{1/4}(\b) \cos(\b).
\]

Thus, plugging the above estimates in \eqref{eq:EE}, we get 
 \beq\label{eq:small2D_th2}
E(0, I_2, I_4) \les \al^{-1}(  M^{-1/(4\al)} \al^{-5}(\e^2 + \al) 
+  \e^{\al / 6} ) .
\eeq

Finally, we estimate $I_3$, which is similar to that of $I_2, I_4$. We use the following decomposition 
\beq\label{eq:pertb_decomy}
\pa_y(\hat \th - \bar \th) 
= f_y(y) + \int_0^{ x} (\hat \eta_y - \bar \eta_y)(z, y) d z \teq f_y(y) +
x J( \pa_y \hat \eta - \pa_y \bar \eta ).
\eeq
Using $|\al^i D_R^i x| = |D_r^i x|  \les x, |D_{\b}^i x| \les x$, \eqref{eq:Jop} for $J$, 
and \eqref{eq:small1} for $\hat \eta - \bar \eta$, for $i+j \leq 5$, 
we yield 
\[
|D_R^i D_{\b}^j (x J( \pa_y \hat \eta - \pa_y \bar \eta))| 
\les \al^{-5} x J(P_3),
\quad P_3 \teq \f{1}{r} \cdot 
\e^{\al / 6} \min(R^3, R^{-2}) (\cos(\b))^{5\al/6}  \one_{\pi/2-\b \leq 2 \e}.
\]

Using estimates similar to \eqref{eq:Jop_eta3}, we derive 
\[
\bal
 P_3(z, y) &\les \e^{\al / 6} \f{1}{y} \min(y^{3\al}, y^{-2\al}) (z/ y)^{5\al / 6} \one_{z \leq 4\e y},  \\
  x J(P_3) &= \int_0^{\min(x, 4 \e y)} P_3(z, y) d z
\les \e^{\al /6} \min(y^{3\al}, y^{-2\al}) \f{ \min(x, y)^{5\al / 6 + 1}}{ y^{1 + 5\al / 6}} \\
&\les \e^{\al / 6} \min(R^{3} (\sin \b)^{3\al}, R^{-2} (\sin \b)^{-2\al}) 
  \cos \b.
\eal
\]

For $f_y(y)$, using \eqref{eq:pertb_f}, chain rule $D_R = \al^{-1} D_r$, 
and $D_R q(y) = 2 \cos^2(\b) (y\pa_y) q(y)$ \eqref{eq:Dy},
for $i+j \leq 5$, we get
\[
|D_R^i D_{\b}^j f_y(y) | \les \al^{-5} ( \cos \b)^{l(j)} \sum_{k\leq 5} |(y\pa_y)^k f_y(y)|
\les \al^{-5} \one_{y\leq 2\d} \al^2 y^{2\al} ( \cos \b)^{l(j)}, 
\quad l(j) = \one_{j \geq 1}.
\]

Since we choose $ \d < \f{1}{2}$ and $ 1 \leq y^{-2\al} = R^{-2}  (\sin \b)^{-2\al}$, we get 
\[
y^{2\al}\one_{y \leq 2\d} 
\les \d^{\al / 4} \min(y^{ 7\al /4}, y^{-2\al})
= \d^{\al/4}\min( R^{7/4} (\sin \b)^{7\al / 4}, R^{-2} (\sin \b)^{-2\al} ).
\]

Combining the above estimates on $f_y(y), J(\pa_y(\hat \eta - \bar \eta)$, and \eqref{eq:pertb_decomy},
we obtain 
\[
|D_R^i D_{\b}^j ( \pa_y(\hat \th - \bar \th))|
\les (\e^{\al / 6} + \al^{-3} \d^{\al / 4})
\min( R^{7/4} (\sin \b)^{7\al / 4}, R^{-2} (\sin \b)^{-2\al} ) (\cos \b)^{l(j)}, 
\quad l(j) = \one_{j \geq 1} ,
\]
which along with $|D_R^i \chi(R / M)| \les_i 1,  D_{\b} \chi(R/ M) = 0$ and \eqref{eq:pertb_decom} imples 
\[
|D_R^i D_{\b}^j I_3 | \les_{i,j} (\e^{\al / 6} + \al^{-3} \d^{\al / 4})
\min( R^{7/4} (\sin \b)^{7\al / 4}, R^{-2} (\sin \b)^{-2\al} ) (\cos \b)^{l(j)}, 
\quad l(j) = \one_{j \geq 1}. 
\]

From the definition of $\cH^3( \psi)$ in \eqref{norm:H22}, the angular weight is not singular near $\b=0$ in the sense that $ \sin(\b)^{-2\al} \in L^2( \sin(\b)^{-\s}  )$. 
for $M \geq  1$, we get 
 \beq\label{eq:small2D_th3}
\bal
|| I_3||_{\cH^3(\psi)}
& \les (\al^{-3} \d^{\al/4} + \e^{\al / 6})
\int \B\{ \f{(1+R)^4}{R^4} (\sin \b)^{-\s} \B(  (\cos \b)^{-\g}  (\cos \b)^2
+ (\cos \b)^{-\s} \B) \\
& \qquad  \cdot \min( R^{7/4} (\sin \b)^{7\al / 4}, R^{-2} (\sin \b)^{-2\al} )^2 \B\} d R d \b  \\
 & \les \al^{-3} \d^{\al/4} + \e^{\al / 6}.
\eal
\eeq

Similarly, using the definition of $\cC^1$ \eqref{norm:H22}, we get $|| I_3 ||_{\cC^1} \les (\al^{-3} \d^{\al/4} + \e^{\al / 6})$. Summarizing \eqref{eq:small2D_om}, \eqref{eq:small2D_th1},\eqref{eq:small2D_th2}, \eqref{eq:small2D_th3}, we prove 
\[
\limsup_{M \to \infty} E(\td \om_0, \td \eta_0, \td \xi_0) 
\les \al^{-3} \d^{\al/4} + \e^{\al / 6} 
+ \al^{-1} \e^{\al / 6} + \al^{5/2}. 
\]
Note that for the $|| \td \xi ||_{\cC^1}$ norm in \eqref{eq:EE}, we have an extra factor $\al^{1/2}$. Choosing $\al$ small enough and then choosing $\d, \e$ small depending on $\al$ and $M$ sufficiently large, we obtain 
\[
E(\td \om_0, \td \eta_0, \td \xi_0)  < K_* \al^2,
\]
where $K_*$ is the constant in Theorem \ref{thm19:bous}. Using the argument below Theorem \ref{thm19:bous}, we conclude the proof of Theorem \ref{thm:bous}. 
\end{proof}

The proof of Theorem \ref{thm:euler_bd} for Euler equations follows essentially the same construction and Proposition \ref{prop:euler_bd}. In \cite{chen2019finite2}, the support of the blowup solution $(\om^{\th}, u^{\th})$ to \eqref{eq:euler1},\eqref{eq:euler2} is controlled uniformly up to the blowup time with $\supp(\om^{\th}), \supp(u^{\th}) \subset \{ (r, z): (r-1)^2 + z^2 < 1/4 \} $, which satisfies the assumption \eqref{eq:euler_supp} in Proposition \ref{prop:euler_bd}. 
Thus, we omit the proof.

 \begin{remark}

There is an oversight in the proof of \eqref{eq:small2D_th0} in Appendix A.6 \cite{chen2019finite2}. The following estimates are established in \cite{chen2019finite2}: $| D_R^i D_{\b}^j J(\bar \eta) | \les J(\bar \eta) \les \bar \eta.$

However, due to the singular weight $ (\sin 2 \b)^{-\g/2}$ in $\cH^3$ \eqref{norm:H22}, \eqref{wg}, for $j \geq 1$, the above estimate does not imply that $(D_R^i D_{\b}^j J(\bar \eta) )^2  (\sin 2 \b)^{-\g} $ is integrable in $\b$ near $\b=0$. Instead, based on Lemma \ref{lem:profile_2D} and the formula of $\bar \eta$ \eqref{eq:profile_2D}, 
it is not difficult to obtain stronger estimates  
\[
| D_R^i D_{\b}^j J(\bar \eta) | \les  J( (\al \sin \tau)^{l(j)} \bar \eta)
+ (\sin \b)^{l(j)} J(\bar \eta), 
\quad l(j) = \one_{j \geq 1},  \quad J( \sin(\tau) \bar \eta ) \les \al^{-1/2} ( \sin(\b) )^{1/3} \bar \eta,
\]
which justifies the integrability in $\b$ near $\b=0$, where $\sin(\tau) \bar \eta(z, y) = \f{y\bar \eta(z, y) }{(y^2 + z^2)^{1/2}} $ \eqref{eq:polar_2D}, \eqref{eq:profile_2D}.


 \end{remark}

\section{Continuation of $C^{\infty}( D \bsh  O ) $ regularity}\label{sec:BKM}

In this section, we prove that the $C^{\infty}(D \bsh O)$ regularity can be continued in various settings. We first introduce weighted H\"older norms for the energy estimates. 

\vspace{0.1in}
\paragraph{\bf{Weighted H\"older norms}}
Denote by $|| f ||_{ \dot C^{\al}(D) }$ the homogeneous H\"older space
\[
 || f ||_{ \dot C^{\al}(D) } \teq \sup_{x, y \in D}  \f{ |f(x) - f(y) | }{|x-y|^{\al}},
 \quad  || f ||_{  C^{\al}(D) }
  =  || f ||_{ \dot C^{\al}(D) } + || f ||_{L^{\infty}}. 
\]
To simplify the notation, we identify $ || f ||_{\dot C^0} $ as $|| f ||_{L^{\infty}}$. Without specification, we choose $D = \R^2_+$ and simplify $|| f ||_{ \dot C^{\al}(\R^2_+) }$ as  $|| f ||_{ \dot C^{\al} }$. Since the solution is not smooth near $0$, we consider the following weighted H\"older space for $\s \geq 1, k \in \BZ_+ \cup \{ 0\}, \al \in (0, 1)$ 
\beq\label{norm:Xk}
\bal
& \la x \ra_{\s} \teq  |x| \cdot \f{ |x|^{\s-1}}{1 + |x|^{\s-1}}, 
\quad  || f ||_{  \dot X_{\s}^{k, \al} } 
 = || \la x\ra_{\s}^{k+\al} \na^k f ||_{\dot C^{\al}} ,  \\
  & || f ||_{X_{\s}^{k, \al}} 
\teq \sum_{ 0 \leq i \leq k} ( || f ||_{ \dot X_{\s}^{i, \al}} + || \la x \ra_{\s}^i \na^i  f ||_{L^{\infty} } ), \quad X_{\s}^{\infty} \teq \bigcap_{k\geq 0}  X_{\s}^{k,\al} .
\eal
\eeq
In Lemma \ref{lem:interp}, we show that the definition of $X_{\s}^{\infty}$ is independent of $\al$. The initial data, e.g. $\om_0$, in Theorems \ref{thm:euler_3D}-\ref{thm:euler_bd} is quite singular near $0$. For each derivative $\pa_i \om_0$, we need to multiply by a weight vanishing more than $|x|$ for the boundedness of $|x|^{\s} \pa_i \om_0$ near $0$. See discussion in the paragraph \textit{Ideas of the proof} in Section \ref{sec:result}. From the definition of $X_{\s}^{k, \al}$,  $f \in X_{\s}^{k, \al}$ implies that $f$ is smooth away from $x=0$. Thus $ X_{\s}^{\infty} \subset C^{\infty}(\R^2_+ \bsh \{0\}) $. Clearly, we have $C_c^{\infty} \subset X_{\s}^{\infty}$.


Note that a similar weighted H\"older space has been used in \cite{elgindi2018finite}. 
In \cite{elgindi2018finite},  
only the lower order norms are used, e.g. $\dot X_1^{0, \al}$ for the vorticity.
Different from  \cite{elgindi2018finite}, we need to perform estimates, including the product rules and elliptic estimates, in $X_{\s}^{k,\al}$ for all $k$. Moreover, 
our weights are not homogeneous for $\s > 1$.

We have the following BKM-type continuation criterion \cite{beale1984remarks,chae1999local} for the Boussinesq equations. 

\begin{prop}\label{prop:BKM_bous}

Let $X_{\s}^{\infty}$ be the space defined in \eqref{norm:Xk}. For any $\s \geq 1, \al \in (0, 1), q \in [1, \infty)$ and initial data $\om_0  \in  X_{\s}^{\infty} \cap C^{\al} \cap L^q, \th_0 \in C^{1,\al}  \cap X_{\s}^{\infty}  , \na \th_0 \in L^q $, $\uu_0, \th_0 \in L^2$, $\om_0$ odd in $x$,$\th_0$ even in $x$,  the unique local solution $\om, \na \th $ to the Boussinesq equations \eqref{eq:bous} in $\R^2_+$ exists in  $L^{\inf}( [0, T); X_{\s}^{\infty} \cap C^{\al} \cap L^q )$ if and only if 
\[
  \int_0^T |  || \na \th(t) ||_{\infty} d t < + \infty.
  \]
Moreover, in the life span, we have $ \th, \uu \in   C^{\infty}(\R^2_+ \bsh \{ 0 \} )$. 

\end{prop}

The symmetries of $\om, \th$ in $x$ are preserved and used to obtain $u_i |_{x_i=0} = 0$
in Section \ref{sec:EE}. 
For the Euler equations \eqref{euler} in $\R^3$ (not neccessary axisymmetric), we have similar results.  

\begin{prop}\label{prop:euler_R3}

For any $\s \geq 1, \al \in (0, 1), q \in [1, \infty)$ with initial data $\om_0 \in C^{\al} \cap X_{\s}^{\infty} \cap L^q$, $\om_{0,i}$ even in $x_i$ and odd in $x_j, i, j \in \{1,2,3\}, i \neq j $, $\uu_0 \in L^2$, the unique local solution $\om, \uu$ to the Euler equations \eqref{euler} exists with $\om, \na \uu 
\in L^{\inf}( [0, T); X_{\s}^{\infty} \cap C^{\al} \cap L^q )$ if and only if 
\beq\label{eq:BKM}
  \int_0^T || \om (t) ||_{\infty} dt < + \infty.
\eeq
Moreover, in the lifespan, we have $\uu(t) \in C^{\infty}(\R^3 \bsh \{ 0 \} ) \cap L^2$. 
\end{prop}




Note that in \eqref{euler}, $\om = \na \times \uu$ is a vector. The above symmetries imply that $u_i(x)$ is odd in $x_i$ and even in $x_j, j \neq i$. These symmetries are preserved by the equations. Note that the vorticity $\om_0 = \om_0^{\th} (-\sin \th, \cos \th, 0)$ in the axisymmetric Euler equations without swirl and with $\om_0^{\th}(r, z)$ odd in $z$ satisfies the symmetries in Proposition \ref{prop:euler_R3}.
Denote  $O = \{ (r, z): r= 1, z=0\}$ and 
\beq\label{eq:smax}
D = \{ (r, z) : r \in [0,1], z \in \BT \},  \quad 
	 S_{max} \teq \{ (r, z) : |(r, z) - (1, 0)| < 1/2 \}.
\eeq

For 3D Euler \eqref{eq:euler1}-\eqref{eq:euler2} in $D$, we have a similar result.

\begin{prop}\label{prop:euler_bd}

Let $X_{O, \s}^{\infty}$ be the space defined in \eqref{norm:XkO}. For any $\s \geq 1, \al \in (0, 1)$ with initial data $\om_0^{\th} \in C^{\al} \cap X_{O, \s}^{\infty} , u_0^{\theta} \in C^{1, \al} , \na u_0^{\theta} \in X_{O, \s}^{\infty}$, 
$\om_0^{\th}$ odd in $z$, $u_0^{\th}$ odd or even in $z$, and $\supp(\om_0^{\th}),\supp(u_0^{\th})\subset S_{max}$,
the unique local solution $\om^{\th} \in L^{\infty}([0, T), C^{\al}),  \uu \in L^{\infty}([0, T), C^{1,\al})$ to the Euler equations  \eqref{eq:euler1}-\eqref{eq:euler2} exists with 
$\om^{\th}(t), \na u^{\th}(t) \in  X_{O, \s}^{\infty} $ if  $\int_0^T ||\om^{\th}(t) ||_{L^{\infty}} + || \na u^{\th}(t)||_{L^{\infty}} dt < +\infty$
and 
\beq\label{eq:euler_supp}
 \cup_{0\leq t \leq T} \supp( \om^{\th}(t) ) \cup \supp( u^{\theta}(t)) \subset S_{max}. 
\eeq
Moreover, for $t < T$, we have $\uu(t) \in  C^{\infty}( S_{max} \bsh O ) \cap L^2$.


\end{prop}

Note that the above symmetries of $\om, u^{\th}$ in $z$ are preserved, 
and $u^{\th}$ can be either odd or even in $z$. 
Since the support is bounded \eqref{eq:euler_supp}, 
the $L^{\infty}$ boundedness implies  $\om^{\th}, \na u^{\th} \in  L^q$ for $q \in [1,\infty]$.
 We will focus on the proof of Proposition \ref{prop:BKM_bous}. The proof of Proposition \ref{prop:euler_R3} is similar and simpler since there is no boundary. Thus, we omit its proof.
In Proposition \ref{prop:euler_bd}, we impose 
\eqref{eq:euler_supp} so that the solution does not touch the axis $r = 0$ and reduce the technicality. The singular solution constructed in \cite{chen2019finite2} satisfies \eqref{eq:euler_supp} up to the blowup time. See Section 9.3.5 \cite{chen2019finite2}.

\subsection{Weighted H\"older estimates}


The weight $\la x \ra_{\s}$ \eqref{eq:xwg} is increasing in $|x|$, and behaves like $|x|$ for large $|x|$. We have simple estimates for $\la x \ra_{\s}$ similar to $|x|$
\beq\label{eq:xwg}
\bal
 & |\na \la x \ra_{\s} | \les_{\s} 1, \quad \la x \ra_{\s} \leq |x|, \quad  
 | \na \la x \ra_{\s}^{l} | \les_{\s, l} \la x \ra_{\s}^{l-1}, \ l \in \R,  \\
 &  
|| \la x \ra_{\s}^{\al} ||_{\dot C^{\al}} \les 1, 
 \quad  |x| \asymp  |y|  \Rightarrow \la x \ra_{\s}  \asymp_{\s}   \la y \ra_{\s}
 \eal
\eeq
where $A \asymp B$ means that there exists absolute constant $C_1, C_2$ such that 
$C_1 A \leq B \leq C_2 A$.


We need several (weighted) H\"older estimates. Firstly, we have the following interpolation.

\begin{lem}\label{lem:interp}
Let $D = \R^d$ or $\R^2_+$. For $\s \geq 1, k, l \in \BZ_+ \cup \{ 0 \}, \al , \b \in [0, 1 )$ with $k+ \al < l + \b$, we have 
\begin{align}
|| \na^k f ||_{ \dot C^{\al}(D) }
& \les_{ k, l, \al, \b} || f ||^{1-m}_{L^{\infty}(D)}
 || \na^l f ||^m_{ \dot C^{\b}(D)}, \quad  m = \f{k+\al}{ l + \b}, \label{eq:interp_nowg}  \\
 ||  f ||_{ \dot X_{\s}^{k, \al}(D) }
& \les_{\s, k, l, \al, \b} || f ||^{1-m}_{L^{\infty}(D)}
 || f ||^m_{ X_{\s}^{l, \b}(D)  }, \quad  m = \f{k+\al}{ l + \b}  \label{eq:interp_wg} .
\end{align}
In particular, the definition of space $X_{\s}^{\infty}$ \eqref{norm:Xk} is independent of $\al \in [0, 1)$. 



\end{lem}

Note that in the upper bound of \eqref{eq:interp_wg}, we use the non-homogeneous norm $X_{\s}^{l, \b}$.

\begin{proof} 

\textbf{Step 1: Unweighted estimate}

In the case of the whole space $D = \R^d$, the unweighted interpolation inequalities can be established by decomposing $f$ into the low and the high frequency part and then use Bernstein inequality \cite{bahouri2011fourier}. 
In the case of $D = \R^2_+$, we cannot define the Fourier transform. Instead, we mimic the proof. 

To estimate the lower frequency in $\na^k f(x)$ for $x \in \R^2_+ $, we want to average it in a region $R(x, r) \subset \R^2_+$. We consider a $C_c^{\infty}(\R)$ function $\chi_0$ supported in $[-1, 0]$ with 
\[
  \int_{\R} \chi_0 = 1, \quad  \int_{\R} \chi_0(y) y^i d y =0, \quad i =0,1,2.., l.
\]
Consider $\td \chi_0(y) \teq 2  \chi_0(2 y) - \chi_0(y) $. The above constraints imply that 
$\int_{\R}  \td \chi_0 y^i = 0,  0\leq i \leq l $.
Thus, we have
$\td \chi_0(y) = \pa_y^l F$ for 
\[
F(z) = \int_{-1}^{z} \int_{-1}^{x_2} ... \int_{-1}^{x_l} \td \chi_0( y) dy ,
\]
and $F \in C_c^{\infty}$ and is supported in $[-1, 0]$. Define  
\[
\chi(x, y) \teq \chi_0(x) \chi_0(y), \quad  \chi_{L}(x) \teq L^2 \chi(L x), 
\quad  \chi_{0, L}(x) = L \chi_0(L x), 
\quad F_L(x) \teq F( L x).
\]

We have 
\[
\bal
 \chi_{2L} - \chi_L 
& =  (\chi_{0, 2L}(x) - \chi_{0, L}(x))  \chi_{0, 2L}(y) 
+ (\chi_{0, 2L}(y) - \chi_{0, L}(y) )  \chi_{0, L}(x)  \\
& = L^{-l} \pa_x^l (  F_L(x)  \chi_{0, 2 L}(y) ) 
+ L^{-l} \pa_y^l (  \chi_{0, L}(x)  F_L(y) ).
\eal
\]

We first consider the case of $\al = 0$. For $R$ to be chosen, we decompose $\na^k f $ as follows \[
  \na^k f = \na^k f \ast \chi_R 
  + \sum_{i \geq 0} \na^k f \ast ( \chi_{2^{i+1} R} - \chi_{ 2^i R} )
  = \na^k f \ast \chi_R 
  + \sum_{i \geq 0} \na^k f \ast 
  ( \chi_{2^{i+1} R} - \chi_{ 2^i R} ) \teq I_{low} + \sum_{i \geq 0} I_i
\]

Formally, the above decomposition projects $\na^k f$ into the low frequency $|\xi| \les R$, and 
$ |\xi| \asymp 2^i R$. For $x \in \R^2_+$, since $\chi(z)$ is supported in $z_i \leq 0 $, the support of the integrand in $ \int \na^k f(y) \chi_L(x-y) d y$ is in $ y_i \geq x_i$ and in $\R^2_+$. 

Using Young's inequality and integration by parts, we can obtain
\[
 |I_{low}|  = |f \ast \na^k \chi_R| \les R^k || f ||_{L^{\infty}}, 
 \quad |I_i | \les (2^i R)^{-l - \b + k } || f ||_{\dot C^{l, \b}}.
\] 

Summing the above estimates and choosing $R = ( \f{ || f ||_{\dot C^{l, \b}}}{ || f ||_{L^{\infty}} })^{1 / (l+ \b)}$ complete the estimates.

For the case of $\al>0$, we have the simple interpolation 
\beq\label{eq:interp_unwg1}
 || \na^k f ||_{\dot C^{\al}} 
 \les  || \na^k f ||_{L^{\infty}}^{1-  \f{\al}{\g} } || \na^k f ||_{\dot C^{\g}}^{ \f{\al}{\g} }
\eeq
for $ \al < \g \leq 1$, where we treat $|| f ||_{\dot C^1} = || \na f ||_{L^{\infty}}$. If $ k < l$, we apply the above interpolation with $\g = 1$ and the $L^{\infty}$ estimates 
\[
 ||\na^k f ||_{L^{\infty}} \les || f ||^{1-m_1}_{L^{\infty}} || f ||_{\dot C^{l, \b}}^{m_1}, \quad 
 ||\na^{k+1} f ||_{L^{\infty}} \les || f ||^{1-m_2}_{L^{\infty}} || f ||_{\dot C^{l, \b}}^{m_2} , \quad 
 m_1 = \f{k}{l+\b},  \ m_2 = \f{k+1}{l + \b}
 \]
 to obtain the desired estimate \eqref{eq:interp_nowg}. If $k = l$, then we have $\al < \b$. 
 Applying \eqref{eq:interp_unwg1}  with $\g = \b $ and the above estimate for $ ||\na^k f ||_{L^{\infty}}$, we prove \eqref{eq:interp_nowg}.

\vspace{0.1in}
\textbf{Step 2: Weighted estimates}

Next, we consider the weighted interpolation. Denote $g= \la x\ra_{\s}^{k+\al} \na^k f$. We first consider the case of $\al = 0$. Let $\chi \in C_c^{\infty}(\R), \chi:  \R \to [0, 1]$ 
with $\chi(r) = 0$ for $r \leq 1$ and $\chi(r) = 1$ for $r \geq 2$. We assume that $|x| \in [2^n, 2^{n+1}]$ and denote $\chi_1(|x|) = \chi( \f{|x|}{2^{n-2}})$. We have 
\beq\label{eq:g_wg}
\chi_1(|x|) = \chi( \f{|x|}{2^{n-2}}), \quad  g = \la x\ra_{\s}^k \na^k (\chi_1(|x| ) f ) ,
\eeq
and $g$ is supported away from $x=0$. Below, the implicit constants can depend on $\s, k, l ,\al, \b$.

Using \eqref{eq:interp_nowg}, we yield 
\beq\label{eq:interp_lei0}
  |g | \les \la x\ra_{\s}^k || \na^l ( f \chi_1) ||_{ \dot C^{\b}}^{m} || f \chi_1 ||^{1-m}_{L^{\infty}}
  \les  \la x\ra_{\s}^k || \na^l ( f \chi_1) ||_{ \dot C^{\b}}^{m} || f  ||^{1-m}_{L^{\infty}},
  \quad m = \f{k}{l+ \b}. 
\eeq

Using Leibniz's rule, we get 
\beq\label{eq:interp_lei1}
\bal
|| \na^l( f \chi_1) ||_{ \dot C^{\b}}
& \les \sum_{i\leq l} ||  \la x\ra_{\s}^{i-l -\b} \na^i \chi_1 \cdot  \la x\ra_{\s}^{l-i + \b}\na^{l-i} f ||_{ \dot C^{\b}} \\
 \eal
\eeq

For H\"older estimate with $y, z \in \supp(\chi_1)$, we consider $|x| \les |y| \leq |z|$. Using decomposition 
\[
\bal
& | \la y \ra_{\s}^{-a}  pq(y) -  \la z \ra_{\s}^{-a}  pq(z) | 
 \leq | \la z \ra_{\s}^{-a}  p(z) (q(y) -q(z)) |   + |q(y) p(y) ( \la y \ra_{\s}^{-a} - \la z \ra_{\s}^{-a})| \\
& \qquad  + | q(y) (p(y) - p(z))  \la z \ra_{\s}^{-a} | \teq I_1 + I_2 + I_3 , \\
 &  a = -( i - l - \b) > 0, \quad p = \na^i \chi_1,  \quad  q(s) = \la s \ra_{\s}^{l-i + \b} \na^{l-i } f, 
\eal
\]
 $ \chi_1 =  \chi( x / 2^{n-1})$,  \eqref{eq:xwg}, $ \la x \ra \leq |x| \asymp 2^n$, and that $\la x \ra_{\s}$ is increasing, we get 
\[
\bal
& |I_1| 
\les |y-z|^{\b} \la x \ra_{\s}^{-a} 2^{- i n} 
|| f ||_{X_{\s}^{l, \b}},  \quad 
|I_2 | \les |y-z|^{\b}  \la y \ra_{\s}^{-a - \b} 
2^{- n i} \cdot \la y \ra_{\s}^{\b} || f ||_{X_{\s}^{l, \b}} \\
&  |I_3| \les |y-z|^{\b} 2^{ - (i+\b) n}
\la y \ra_{\s}^{\b} 
\la z \ra_{\s}^{-a}
\les |y-z|^{\b} 2^{- (i+\b) n} \la y \ra_{\s}^{-a + \b},
\eal
\]
where we have used $|q(y)| \leq \la y \ra_{\s}^{\b} || \la s \ra_{\s}^{l-i} \na^{l-i} f ||_{L^{\infty}}\les \la y \ra_{\s}^{\b} || f ||_{X_{\s}^{l, \b}}$. 
Using $|x| \les |y|, -a + \b \leq 0,  |x| \asymp 2^n, \la x \ra_{\s} \leq |x| \les 2^n$ \eqref{eq:xwg}, $i+ a = l + \b$, and combining the above estimates, we establish 
\beq\label{eq:interp_lei2}
\bal
I_1 + I_2 + I_3 
& \les |y-z|^{\b} || f ||_{X_{\s}^{l, \b}} ( \la x \ra_{\s}^{-a} \la x \ra_{\s}^{-i}
 + \la x \ra_{\s}^{-a} \la x \ra_{\s}^{-i}
 + \la x \ra_{\s}^{-a + \b} \la x \ra_{\s}^{-(i+\b)} )  \\ 
 &  \les |y-z|^{\b} || f ||_{X_{\s}^{l, \b}} \la x \ra_{\s}^{- i - a} 
 \les |y-z|^{\b} || f ||_{X_{\s}^{l, \b}} \la x \ra_{\s}^{- l  - \b},   \\
  || \na^l( f \chi_1) ||_{ \dot C^{\b}}
& \les 
  || f ||_{X_{\s}^{l, \b}} \la x \ra_{\s}^{- l  - \b} .
\eal
\eeq


Combining the above estimates, we prove the desired estimate \eqref{eq:interp_wg} with $\al = 0$. 

Denote $g = \la x\ra_{\s}^{k+\al} \na^k (\chi_1(|x| ) f )$. For $\al > 0$, we estimate the H\"older norm $\f{|g(x) - g(y)|}{|x-y|^{\al}}$ with $|x| \leq |y|$. 
We consider two cases. If $|y| \geq 2 |x|$, using $|x-y|^{\al} \gtr |y|^{\al} 
\gtr |x|^{\al} + |y|^{\al} $, $\la x \ra_{\s} \leq |x|$, and  
\eqref{eq:interp_wg} with $\al = 0$, we yield 
\[
 \f{ |g(x) -g(y)|}{|x-y|^{\al}}
\les \max\B( \f{|g(x)|}{|x|^{\al}},  \f{|g(y)|}{|y|^{\al}} \B)
\les || \la x\ra_{\s}^k \na^k f ||_{L^{\infty}}
\les || f||_{L^{\infty}}^{ 1-m_1}
|| f||_{X_{\s}^{l,\b}}^{m_1}, \quad m_1 = \f{k}{ l+ \b}. 
\]
Since $m_1 < m= \f{k+\al}{l+\b}$ and $|| f||_{L^{\infty}}
\leq || f||_{X_{\s}^{l,\b}}$, we can further bound the upper bound by $|| f||_{L^{\infty}}^{ 1-m}|| f||_{X_{\s}^{l,\b}}^{m}$. 

If $|y| \in [|x|, 2 |x|]$, we assume that $|x| \in [2^n, 2^{n+1}]$. We introduce the  same cutoff function $\chi_1$. 
Using \eqref{eq:xwg}, we have 
\[
\bal
 I & \teq \f{ |g(x) - g(y)|}{|x-y|^{\al}}
= \f{  | \  \la x\ra_{\s}^{k+\al} \na^k( f \chi_1)(x) -\la y\ra_{\s}^{k+\al} \na^k( f \chi_1)(y)  |  }{|x-y|^{\al}} \\
& \les \la x\ra_{\s}^k || \na^k( f \chi_1)(x) ||_{L^{\infty }}
+ \la x\ra_{\s}^{k+\al} || \na^k( f \chi_1)(x)  ||_{\dot C^{\al}} .
\eal
\]

Using interpolation \eqref{eq:interp_nowg} and \eqref{eq:interp_lei0}-\eqref{eq:interp_lei2}, we get 
\beq\label{eq:interp_lei3}
\bal
|| \na^k( f \chi_1)(x) ||_{L^{\infty }} 
& \les  
 || f ||_{L^{\infty}}^{1- m_1} || \na^l ( f \chi_1) ||_{ \dot C^{\b} }^{m_1} \les 
 || f ||_{L^{\infty}}^{1- m_1}  ( \la x \ra_{\s}^{- (l + \b)} || f ||_{X_{\s}^{l, \b}} )^{m_1} 
\les \la x \ra_{\s}^{ -  k} || f ||_{L^{\infty}}^{1- m}   || f ||_{X_{\s}^{l, \b}} ^{m},  \\
|| \na^k( f \chi_1)(x) ||_{ \dot C^{\al}} & \les
 || f ||_{L^{\infty}}^{1- m} || \na^l ( f \chi_1) ||_{ \dot C^{\b} }^m \les 
 || f ||_{L^{\infty}}^{1- m}  ( \la x \ra_{\s}^{- (l+ \b)} || f ||_{X_{\s}^{l, \b}} )^{m} 
\les \la x \ra_{\s}^{- (k + \al)}
|| f ||_{L^{\infty}}^{1- m}   || f ||_{X_{\s}^{l, \b}}^{m} ,
\eal
\eeq
where $m_1 = \f{k}{ l + \b} < m = \f{k+\al}{l + \b}$, and we have used $|| f ||_{L^{\infty}} \leq || f||_{X_{\s}^{l ,\b}}, n(l+\b) \cdot m_1 = n k, n(l+\b) m =n(k+\al)$. Combining the above estimates,
we establish \eqref{eq:interp_wg}. 
\end{proof}

We have the following product rule. 
\begin{lem}\label{lem:prod}
For $\s \geq 1, k \in \BZ_+ \cup \{ 0 \}, \al \in [ 0, 1)$, we have
\[
 || \la x\ra_{\s}^{k+\al} \na^k ( f g) ||_{\dot C^{\al}}
\les_{\s, k, \al} || f ||_{L^{\infty}} || g ||_{X_{\s}^{k, \al}}
+  || g ||_{L^{\infty}} || f ||_{X_{\s}^{k, \al}}.
\]
\end{lem}

\begin{proof}
Below, the implicit constants can depend on $\s, k, \al$. Using a direct calculation yields 
\[ || \la x \ra_{\s}^{k+\al} \na^k ( f g) ||_{ \dot C^{\al}}
\les \sum_{i \leq k} || \la x\ra_{\s}^{\al}\cdot \la x\ra_{\s}^i \na^i f \cdot  \la x\ra_{\s}^{k-i} \na^{k-i} g  
||_{ \dot C^{\al}} .
\]
Using the following identity with $p = \la x \ra_{\s}^i \na^i f, q = \la x \ra_{\s}^{k-i} \na^{k-i} g $, and triangle inequality, we yield 
\[
\bal
& \la x \ra_{\s}^{\al}  pq(x) - \la y \ra_{\s}^{\al} p q(y)
= (\la x \ra_{\s}^{\al} p(x)   - \la y \ra_{\s}^{\al} p(y)) q(x)
+ ( \la y \ra_{\s}^{\al} - \la x \ra_{\s}^{\al}) p(y) q(x)
+ (\la x \ra_{\s}^{\al} q(x) - \la y \ra_{\s}^{\al} q(y)) p(y), \\
& || \la x \ra_{\s}^{\al}\cdot \la x \ra_{\s}^i \na^i f \cdot  \la x \ra_{\s}^{k-i} \na^{k-i} g  
||_{ \dot C^{\al}}
 \les 
|| \ \la x \ra_{\s}^{i} \na^i f ||_{L^{\infty}} 
|| \ \la x \ra_{\s}^{k-i+\al} \na^{k-i} g   ||_{ \dot C^{\al}} \\
& \qquad  \qquad + || \ \la x \ra_{\s}^{i+\al} \na^i f ||_{\dot C^{\al} } 
|| \ \la x \ra_{\s}^{k-i} \na^{k-i} g   ||_{L^{\infty}}  
+  || \la x \ra_{\s}^{i} \na^i f ||_{L^{\infty}} || \la x \ra_{\s}^{k-i} \na^{k-i} g   ||_{L^{\infty}}
\teq I_1 + I_2 + I_3. 
\eal
\]


Let $p_i = \f{i}{k+\al}, q_i = \f{ i+\al}{k+\al}$. By definition, we get $p_i + q_{k-i} = 1$. Using Lemma \ref{lem:interp} and Young's inequality, we get 
\[
\bal
I_1 & \les 
|| f ||_{L^{\infty}}^{1 - p_i} 
|| f ||_{X_{\s}^{k, \al}}^{p_i}
 || g ||_{L^{\infty}}^{1 - q_{k-i}}  || g ||_{X_{\s}^{ k , \al} }^{q_{k-i}}
 = || f ||_{L^{\infty}}^{1 - p_i} 
|| f ||_{X_{\s}^{k, \al}}^{p_i}
 || g ||_{L^{\infty}}^{ p_i }  || g ||_{X_{\s}^{ k , \al} }^{1 - p_i} \\
 & \les || f ||_{L^{\infty}} 
 || g ||_{X_{\s}^{k, \al}} + || g ||_{L^{\infty}} 
 || f ||_{X_{\s}^{k, \al}} .
 \eal
\]
The estimates of $I_2, I_3$ and other terms are similar. We complete the proof.
\end{proof}

\subsection{Weighted estimates of the velocity}
Next, we estimate the velocity $\uu = \na^{\perp}(-\D)^{-1} \om$ in $\R_2^+$. Denote by $\psi(f)$ the stream function associated with $f \in \R^2_+$
\[
- \D \psi = f, \quad \psi(x, 0) = 0.
\]
Let $F$ be the odd extension of $f$ from $\R^2_+$ to $\R^2$. We have
\beq\label{eq:BSlaw_form}
\bal
& \pa_x^i \pa_y^j(-\D)^{-1} \psi 
=  c_{ij} f + C_{ij} P.V. \int K_{ij}(x- y) F(y) d y, \quad i+j=2, \\
& K_{11} = \f{y_1 y_2}{|y|^4}, \quad K_{20} = K_{02} = \f{y_1^2 - y_2^2}{|y|^4}, \quad
c_{11} = 0,  
\eal
\eeq
for some universal constants $c_{ij}, C_{ij}$. 
The following estimates are standard 

\begin{lem}\label{lem:BSlaw0}
For $i+j = 2$, $\al \in (0, 1)$, $q \in [1, \infty)$, and $R >0$, we have
\[
\bal
  & || \pa_x^i \pa_y^j (-\D)^{-1} f ||_{\dot C^{\al}}
  \les_{\al} || f ||_{\dot C^{\al}}, \\
  &  || \pa_x^i \pa_y^j (-\D)^{-1} f ||_{L^{\infty}}
\les_{\al }   || f ||_{\dot C^{\al}} R^{\al}
+ | \int_{|y|_{\infty} \geq R} K_{ij}( y) F(x- y)  d y | 
\les_{\al, q} || f ||_{\dot C^{\al}} + || f ||_{L^q}.
\eal
\]
\end{lem}

Next, we develop the weighted estimate for the velocity. 
\begin{lem}\label{lem:BSlaw}
For $\s \geq 1, k \in \BZ_+ \cup \{ 0 \}, \al \in  ( 0, 1), q \in [1, \infty) $, we have
\[
 || \ \la x \ra_{\s}^{k+\al} \na^{k+1}   \uu ||_{  \dot C^{\al}}
 +  || \la x \ra_{\s}^{k}  \na^{k+1} \uu ||_{ L^{\infty}}
 \les_{\s, k, \al, q} || \om ||_{ X_{\s}^{k, \al}} +  || \om ||_{L^q} + || \om ||_{\dot C^{\al}} 
\]
\end{lem}

\begin{proof}

Below, the implicit constants can depend on $k,\al, \g$. We fix $k$ and denote
\beq\label{eq:BSlaw_EE}
|| \om ||_X \teq  || \om ||_{ X_{\s}^{k, \al}} +  || \om ||_{L^q} + || \om ||_{\dot C^{\al}} .
\eeq

Denote $\psi = (-\D)^{-1} \om$. Since $\uu = \na^{\perp} \psi$, we need to estimate $ \pa_x^i \pa_y^j \psi$ for  $i+j = k+2$. 
Note that the nonlocal operator $(-\D)^{-1}$ commutes with $\pa_x$ but does not commute with $\pa_y$. Using the relation $ \pa_y^2\psi = - \pa_x^2 \psi - \om$ and 
\[
||  \ \la x\ra_{\s}^{k+\al} \pa_x^i \pa_y^j \om  ||_{\dot C^{\al}} 
+ || \ \la x \ra_{\s}^{k} \pa_x^i \pa_y^j \om  ||_{L^{\infty}} \les || \om ||_{X_{\s}^{k, \al}},  \quad i+ j = k,
\]
we only need to perform weighted estimate $\pa_x^i \pa_y^j \psi$ for $i+j = k+2, j \leq 1$. 
Denote 
\[
 \pa_x^i \pa_y^j (-\D)^{-1} \om =\pa_x^k R_{ij} \om, \quad R_{ij} \teq \pa_x^{i-k} \pa_y^j (-\D)^{-1} .
\]


From the proof of Lemma \ref{lem:interp}, we only need to estimate 
\beq\label{eq:BSlaw_terms}
 I = || \la x \ra_{\s}^k g ||_{ L^{\infty}}, \quad 
 II = \f{ | \la x \ra_{\s}^{k+\al} g(x) -\la z \ra^{ k+ \al} g(z)|}{|x-z|^{\al}} , \quad |z| \in [|x|, 2|x|], \quad g =   \na^{k+1} \uu.
\eeq
For $II$, we can bound it by 
\beq\label{eq:BSlaw_terms1}
|II| \les \la x\ra_{\s}^k |g(x)| + \la x\ra_{\s}^{k+\al} \f{ |g(x) -g(z)|}{|x-z|^{\al}} .
\eeq


We fix $(x, z)$ with  $|x| \in [2^n, 2^{n+1}]$. We construct cutoff function $\chi$ with $\chi(s) = 1$ for $s \leq 1$ and $\chi = 0$ for $s \geq 2$. We assume $|x| \in [2^n, 2^{n+1}]$. Then we define 
\[
\chi_1(x) = \chi( \f{|x|}{2^{n-2}}), \quad \chi_2(x) = 1 - \chi( \f{|x|}{2^{n-2}}).
\]

Since $\pa_x^k$ commutes with $R_{ij}$, we decompose $\om $ and $\pa_x^k R_{ij} \om$ as follows 
\beq\label{eq:BSlaw_terms2}
\bal
\om  & = \om_1 + \om_2 , \quad \om_i = \om \chi_i, 
\quad 
\pa_x^k R_{ij} \om 
 = \pa_x^k R_{ij}  \om_1 +   R_{ij} \pa_x^k \om_2 .
\eal
\eeq

Note that $\om_1, \om_2$ are supported near $0$ and near $x$, respectively. We will pass the derivative $\pa_x^k$ to $\om_2$ and $\pa_x^k$ to the kernel in $ R_{ij}\om_1$. Let $ W_i$ be the odd extension of $\om_i$ from $\R^2_+$ to $\R^2$.

Firstly, for $k=0$, using Lemma \ref{lem:BSlaw0}, we get 
\beq\label{eq:BSlaw_est0}
 || \na \uu ||_{L^{\infty}} \les || \om ||_X .
\eeq


\vspace{0.1in}
\textbf{Estimate of $\om_2$.}
Using Lemma \ref{lem:BSlaw} with $R = 2^n$, we get 
\[
\bal
    || R_{ij} \pa_{x_1}^k \om_2 ||_{ \dot C^{\al}}  
   & \les || \pa_{x_1}^k \om_2 ||_{ \dot C^{\al}},  \\
  | R_{ij} \pa_{x_1}^k \om_2(x)| 
  & \les || \pa_{x_1}^k \om_2 ||_{\dot C^{\al}}  R^{\al}
  + \B| \int_{|x-y|_{\infty} \geq R} K_{ij}(x-y) (\pa_{x_1}^k W_2)(  y) d y \B| \teq I_1 + I_2.
  \eal
\]

For $I_2$ with $k \geq 1$, using integration by part once, we yield
\[
|I_2| \les 
|| \pa_{x_1}^{k-1} W_2||_{\infty} (  \int_{|x-y|_{\infty} = R} |x-y|^{-2} d y
+ \int_{|x-y|_{\infty} \geq R} |x-y|^{-3} dy ) 
\les R^{-1}|| \pa_{x_1}^{k-1} W_2||_{\infty}.
\]
Optimizing $R$, we obtain 
\[
  | R_{ij} \pa_{x_1}^k \om_2(x)|  \les 
  || \pa_{x_1}^k \om_2 ||_{\dot C^{\al}}^{1 / (1+\al)}
  || \pa_{x_1}^{k-1} \om_2 ||_{ L^{\infty} }^{\al /(1 + \al)}, \quad k \geq 1.
\]



For $y \in \supp(\chi_2)$, 
we have $|y| \gtr 2^n , |x| \asymp 2^n$, Using 
\eqref{eq:interp_lei3} and $|| \om||_{L^{\infty}} \les || \om ||_{X_{\s}^{k, \al}} $, we obtain 
\[
|| \pa_{x_1}^k (\om \chi_2) ||_{\dot C^{\al}}
\les \la x \ra_{\s}^{-(k+\al)} || \om ||_{X_{\s}^{k, \al}} , 
\quad || \pa_{x_1}^{k-1} (\om \chi_2) ||_{L^{\infty}} \les 
\la x \ra_{\s}^{-(k-1)} || \om ||_{X_{\s}^{k, \al}} .
\]



Denote $ f_2 =  R_{ij} \pa_{x_1}^k \om_2$. 
Since $|x| \asymp 2^n, |z| \in [|x|, 2 |x|]$, using the above estimates, we prove 
\beq\label{eq:BSlaw_est1}
\bal
  & \la x \ra_{\s}^k |f_2(x)| =  \la x\ra_{\s}^k | R_{ij} \pa_{x_1}^k \om_2| 
\les || \om ||_{X_{\s}^{k, \al}}
\la x \ra_{\s}^{k - \f{k+\al}{ 1+\al} - \f{\al  (k-1)}{ (1+\al)} }
\les  || \om ||_{X_{\s}^{k, \al}},\quad  k \geq 1, \\
& \la x \ra_{\s}^{k+\al} || f_2 ||_{\dot C^{\al}}
\les \la x \ra_{\s}^{k+\al} 
|| \pa_{x_1}^k (\om \chi_2) ||_{\dot C^{\al}}
\les  || \om ||_{X_{\s}^{k, \al}} , \quad k \geq 0.
\eal
\eeq


\vspace{0.1in}
\textbf{Estimate of $\om_1$.} Recall the decomposition \eqref{eq:BSlaw_terms2}. 
 Since $\om_1$ is supported in $\{ y: |y| \leq 2^{n-1} \}$, away from $x, z$, $|x| , |z| \in [2^n, 2^{n+2}]$, using the formulas \eqref{eq:BSlaw_form}, for $a=x, z$, we get 
\[
\pa_{x_1}^k R_{ij} \om_1( a ) 
= C_{ij} \int  \pa_{x_1}^k K_{ij}(a - y) W_1(y) d y.
\]

For $y \in \supp(\om_1)$, 
we get $ |y| \leq \f{ |x| }{2}, |y-x| \gtr  |x|$. For $k \geq 1$, using 
$\la x \ra_{\s} \leq |x|$ \eqref{eq:xwg}, we yield 
\beq\label{eq:BSlaw_est2}
\bal
& |\pa_{x_1}^k R_{ij} \om_1( x ) | 
 \les || W_1||_{L^{\infty}} \int_{ |y- x| \gtr |x| } |x-y|^{-k-2}  d y  \les || \om ||_{L^{\infty}}  |x|^{-  k} , \\
 & \la x \ra_{\s}^k |\pa_{x_1}^k R_{ij} \om_1( x ) |  
\les || \om ||_{L^{\infty}}   \la x \ra_{\s}^{  k} |x|^{-k}
\les || \om ||_{L^{\infty}}    , \ k \geq 1.
\eal
\eeq

For the H\"older estimate with $k \geq 0$, since $|y|\leq 2^{n-1},  |z| \geq |x|\geq 2^n \geq 2|y|$, 
we yield  $ |x-y| \asymp |x|, |z-y| \asymp |z|$. Using 
 $|\na^l K_{ij}(s)| \les_l |s|^{-l-2}$ and $ \la x \ra_{\s} \leq |x|$ \eqref{eq:xwg}, we obtain
\[
\bal
& |(\pa_{x_1}^k K_{ij})(x- y) - (\pa_{x_1}^k K_{ij})(z- y)| 
\les \min( |x|^{-k-2}, |x-z| |x|^{-k-3})
\les |x-z|^{\al} |x|^{-k-2-\al}, \quad  k \geq 0 \\
& \f{|\pa_{x_1}^k R_{ij} \om_1( x ) - 
\pa_{x_1}^k R_{ij} \om_1( z ) 
 | }{ |x-z|^{\al}}
 \les  \int_{|y| \les |x|} |x|^{-k-2-\al} d y
 || \om||_{L^{\infty}}
 \les  || \om||_{L^{\infty}} |x|^{-  (k+\al)} 
\les || \om||_{L^{\infty}} \la x \ra_{\s}^{-  (k+\al)}  ,
 \eal
 \]

Combining \eqref{eq:BSlaw_est0}-\eqref{eq:BSlaw_est2}, we prove the $L^{\infty}$ estimate in Lemma \ref{lem:BSlaw}. Combining \eqref{eq:BSlaw_est0},\eqref{eq:BSlaw_est1}, 
the above estimates, and using \eqref{eq:BSlaw_terms1}-\eqref{eq:BSlaw_terms2}, we 
prove the H\"older estimate in Lemma \ref{lem:BSlaw}.
\end{proof}

\subsection{Energy estimates}\label{sec:EE}

We are in a position to prove Proposition \ref{prop:BKM_bous}. Let $C(k,  t)$ be a constant depending on $ \int_0^t  || \na \th(s) ||_{\infty} ds$, the initial data,
the order of energy estimates $k$, and parameters $\al, q, \s$. It can vary from line to line. Using standard energy estimates, see e.g. \cite{chae1999local}, we have
\beq\label{eq:EE_linf1}
\sup_{ s \leq t}
|| \na \uu(s) ||_{C^{\al} }  +  || \om(s) ||_{ C^{\al} }  
+ || \th(s) ||_{C^{1, \al}} 
+ || \om(s) ||_{L^q} + || \na \th(s) ||_{L^q} 
+ || \uu(s)||_{L^2}  \leq C(0, t).
\eeq

Using embedding, we yield 
\beq\label{eq:EE_linf2}
|| \uu ||_{L^{\infty}} \les || \uu||_{L^2}  + || \na \uu ||_{L^{\infty}}.
\eeq

Next, we perform weighted estimate. 
Using \eqref{eq:bous}, we derive the equation for $\pa_l \th, l=1,2$
\beq\label{eq:bous_th}
   \pa_t  \pa_l \th  + \uu \cdot \pa_l \th 
   = - \pa_l \uu \cdot \na \th .
\eeq

We focus on the estimate of $\om$ \eqref{eq:bous}, and the estimate of $\na \th$ \eqref{eq:bous_th} is similar. Let $\al \in [0, 1)$ and $\g = 0$ or $\al$. Taking $D_{ij} \teq \la x \ra_{\s}^{i+j } \pa_x^i \pa_y^j$ for $i+j \leq k$ on \eqref{eq:bous} and multiplying $\la x \ra_{\s}^{\g}$, we get
\[
\bal
  \pa_t \la x \ra_{\s}^{\g} D_{ij} \om & + \uu \cdot \na ( \la x \ra^{\g}_{\s} D_{ij} \om ) = 
 \uu \cdot \na( \la x \ra_{\s}^{i+j + \g}) \pa_x^i \pa_y^j \om \\
& + \la x \ra_{\s}^{\g}  \B( D_{ij} \pa_x \th - \sum_{i_1 \leq i, j_1 \leq j, i_1 + j_1 \geq 1}
 C_{i_1 j_1  i j}
 D_{i_1 j_1} \uu \cdot D_{i-i_1, j-j_1} \na \om \B) \teq I_1 + I_2, \\
 \eal
\]
where $I_1$ comes from the commutator 
between advection $\uu \cdot \na$ and the weight $\la x \ra_{\s}^{i+j + \g}$, and  $C_{i_1j_1 ij}$ are some absolute constants. Recall the weighted norm \eqref{norm:Xk}. For $i_1 + j_1 \geq 1, i_1 \leq i, j_1\leq j$, using Lemma \ref{lem:prod}, its proof, and then Lemma \ref{lem:BSlaw}, we yield 
\[
\bal
|| \la x \ra_{\s}^{\g} D_{i_1 j_1} \uu \cdot D_{i-i_1, j-j_1} \na \om ||_{\dot C^{\g}}
& \les_{k, \al}  || \na \uu ||_{L^{\infty}} || \om ||_{X_{\s}^{i+j,\al}}
+ || \om ||_{L^{\infty}} || \na \uu ||_{ X_{\s}^{i+j, \al} } \\
& \leq  C(k, t) || \om ||_{X_{\s}^{i+j,\al}}
+ C(k, t) ( || \om ||_{X_{\s}^{i+j,\al}} + C(k, t) ),
\quad \g = 0, \al ,
\eal
\] 
where we treat $\dot C^{0} = L^{\infty}$ and have used $\g \leq \al$. Denote $\uu = (u_1, u_2)$. Due to symmetry in $x$ and the boundary condition, we have $u_i |_{x_i =0} = 0$. Using \eqref{eq:xwg} and a direct computation, we yield
\[
\bal
& I_1 =  (i+j+ \g) \f{ \uu \cdot  \na \la x \ra_{\s} }{ \la x \ra_{\s}} \la x \ra_{\s}^{\g} D_{ij}\om, \quad 
 \f{ u_l \cdot  \pa_l \la x \ra_{\s}}{ \la x \ra_{\s} }
 = \f{u_l}{x_l} \cdot  \f{ x_l \pa_l \la x \ra_{\s}} { \la x \ra_{\s} }
 =\f{u_l}{x_l} \cdot   \f{x_l^2}{|x|^2} ( \s - (\s-1) \f{|x|^{\s-1}}{ 1 + |x|^{\s-1}} ) ,  \\
 & |\f{u_1}{x_1}(x) - \f{u_1}{z_1}(z) | 
 = \B| \int_0^1 u_{1, x}( t x_1, x_2)  - u_{1, x}(tz_1, z_2) d t | 
 \les |x-z|^{\al} || \na \uu||_{C^{\al}}.
 \eal
\]
Similar H\"older estimates hold for $u_2 / x_2$. 
 Using Lemma \ref{lem:prod} and argument therein, we get 
\[
|| I_1 ||_{\dot C^{\g}} \les || \na \uu||_{C^{\al}} || \om ||_{X_{\s}^{i+j, \al}} 
\les C(k,t)   || \om ||_{X_{\s}^{i+j, \al}} .
\]

Using the characteristic along the flow, derivation for the H\"older estimate, see e.g. Lemma 2.5 \cite{ChenHou2023a}, and the above estimates,  
 we yield 
\[
\bal
 ||  \la x \ra_{\s}^{\g}  D_{ij} \om(T) ||_{ \dot C^{\g}}
 & \leq 
  ||  \la x \ra_{\s}^{\g}  D_{ij} \om_0 ||_{ \dot C^{\g}}
+ \int_0^T  C(k, t) ( || \na \th ||_{ X_{\s}^{i+j, \al}} 
 +   || \om ||_{X_{\s}^{i+j, \al}} + C(k, t) )  dt  ,\quad  \g = 0, \al \\
 \eal
\]

The energy estimates for $ ||\la x \ra_{\s}^{\g} D_{i j} \pa_l \th ||_{\dot C^{ \g}}$ 
are similar. Combining the estimates for $i+j\leq k$, we derive 
\[
 || \om(T) ||_{X_{\s}^{k, \al}} + ||  \na \th(T) ||_{X_{\s}^{k, \al}} ) 
 \leq C_k ( || \om_0 ||_{X_{\s}^{k, \al}} + ||  \na \th_0 ||_{X_{\s}^{k, \al}} ) 
    + \int_0^T C(k, t) ( || \om ||_{X_{\s}^{k, \al}} + || \na  \th ||_{X_{\s}^{k, \al}}  +   C(k, t) )  d t .
\]

Using Gronwall's inequality and Lemma \ref{lem:BSlaw}, we establish
\[
\sup_{s \leq t} ( || \om(s) ||_{X_{\s}^{k, \al}} + ||  \na \th(s) ||_{X_{\s}^{k, \al}} 
+ || \na \uu(s) ||_{X_{\s}^{k, \al}} ) \leq C(k, t).
\]

Since $k$ is arbitrary, combining the above estimate and \eqref{eq:EE_linf1},\eqref{eq:EE_linf2}, 
we conclude the proof.


\subsection{Continuation criterion for Euler equations }

The proof of Proposition \ref{prop:euler_R3} follows that of Proposition \ref{prop:BKM_bous}. 
The proof of Proposition \ref{prop:euler_bd} is similar to that of Proposition \ref{prop:BKM_bous}. We only sketch the difference. Denote 
\[
S(t) = \supp( \om(t) ) \cup \supp( u^{\th}(t)) .
\]

Since we impose the support constraint \eqref{eq:euler_supp} and 
the boundary condition $u^r(1, z) = 0$ \eqref{eq:euler2}, for $(r, z) \in S $, we get 
\beq\label{eq:euler_r}
 r, r^{-1} \in C^{\infty}, \quad \sup_{r, z \in S} | u^r / r |
\les \sup_{r, z \in S} | u^r |  \les \sup_{r, z \in S} | \pa_r u^r|. 
\eeq

Let $\chi$ be a smooth cutoff function with $\chi(s) = 1, s \leq \f{3}{4}, \chi(s) =0, s \geq \f{7}{8} $. Define 
\[
\chi_1(r, z) \teq \chi( |(r, z) - (1, 0)|), \quad \psi_1 \teq \chi_1 \td \psi.
\]

Recall $S_{max}$ from \eqref{eq:smax}. We have $\chi_1 = 1$ in $S_{max}$. 
In the energy estimates, we only need the bounds of $u^r,u^z, \om^{\th}, u^{\th}$ in $S(t)$. Denote $\vartheta = \arctan \f{x_2}{x_1}$. From \eqref{eq:euler2},  we have 
\[
 -\D_{3D} (\td \psi \sin \vartheta) = \om^{\th} \sin \vartheta, \quad |(x_1, x_2)| \leq 1, \quad  z \in \BT,
\]
with Dirichlet boundary condition. Using standard elliptic estimate, we yield 
\[
  || \td \psi \sin \vartheta ||_{C^{2,\al}(D)} \les_{\al} ||  \om^{\th} \sin \vartheta ||_{C^{\al}(D)}.
\]

Viewing $\td \psi, \om^{\theta}$ as functions of $r \leq 1, z \in \R$, from the above estimates, we obtain
\beq\label{eq:iter0}
 || \td \psi ||_{C^{2,\al}(D_1)} \les   || \td \psi \sin \vartheta ||_{C^{2,\al}(D)}
 \les || \om^{\th}||_{C^{\al}}, \quad D_1 \teq \{ |(r, z)-(1,0)|\leq \f{7}{8}\},
\eeq
and $ \supp(\chi_1) \subset D_1$. Multiplying \eqref{eq:euler2} by $\chi_1$, we obtain the equations of $\psi_1$
\beq\label{eq:elli1}
\bal
& - ( \pa_{rr}  + \pa_{zz} ) \psi_1 
= \om^{\theta}   + Z(\td \psi, \chi_1),  \\
& Z(\td \psi, \chi_1) \teq - \f{1}{r^2} \td \psi \chi_1 + \chi_1 \f{1}{r} \pa_r \td \psi 
- 2 \pa_r \td \psi  \pa_r \chi_1 
- \td \psi \pa_r^2 \chi_1 
- 2 \pa_z \td \psi \pa_z \chi_1 - \td \psi \pa_z^2 \chi_1 ,
\eal
\eeq
with boundary conditions $\psi_1(1, z) = 0$, where we have used $\om^{\theta} \chi_1 = \om^{\theta}$ due to \eqref{eq:euler_supp}. The function $\psi_1$ can be viewed as a solution in the full half plane. Thus, we can use the elliptic estimate in Lemmas \ref{lem:BSlaw0}, \ref{lem:BSlaw}. 

We introduce weighted spaces $ d_0(r, z) \teq |( r,z)-(1,0)|,$
\beq\label{norm:XkO}
\bal
 &   \quad d =  \f{d_0^{\s}}{1 + d_0^{\s-1}}, \quad 
 || f ||_{  \dot X^{k, \al} } 
 = || d^{k+\al}(r, z) \na^k f ||_{\dot C^{\al}} ,  \ || f ||_{X_{O,\s}^{k, \al}} 
\teq \sum_{ 0 \leq i \leq k} ( || f ||_{ \dot X_{O,\s}^{i, \al}} + || d^i(r,z) \na^i  f ||_{L^{\infty} } ),
\eal
 \eeq
 similar to \eqref{norm:Xk} and other spaces similarly. 
 Using \eqref{eq:iter0}, we get 
\[
|| Z||_{C^{1,\al}} \les || \om^{\theta}||_{C^{\al}}. 
\]
Applying Lemmas \ref{lem:BSlaw0}, \ref{lem:BSlaw}, we can establish the $X_{O, \s}^{1,\al}$ estimate
for $\na^2 \psi_1$. 
Recall $S_{max}$ from \eqref{eq:smax}.
By choosing the next cutoff function $\chi_2$ with 
\[
S_{max}
	 \subset \subset \{ (r, z) : \chi_2(r, z) = 1 \}, \quad 
	 \supp(\chi_2) \subset \subset \{ (r, z) : \chi_1(r, z) = 1 \},
\]
deriving the equations of $ \td \psi \chi_2$ similar to \eqref{eq:elli1} using \eqref{eq:euler2}, and applying Lemma \ref{lem:BSlaw}, we obtain the $X_{O,\s}^{2,\al}$ estimate for $\na^2 (\td \psi \chi_2)$. We note that since the domain is bounded, we have $|| d^{k+\al + l} f ||_{X_{O,\s}^{k, \al}} \les || d^{k+\al} f ||_{X_{O,\s}^{k,\al}}$. Repeating the above process, we obtain $X_{O,\s}^{k, \al}$ estimate for $\na^2  \td \psi$ in $S_{max}$ for any $k$. Since $u^{\th}$ is supported in $S_{max}$, using Lemma \ref{lem:prod}, we have 
\[ || u^{\th }||_{X_{O,\s}^{k,\al}} 
\les_k || \na u^{\th} ||_{X_{O,\s}^{k-1,\al}}, \quad  || \na ( (u^{\th})^2) ||_{X_{O,\s}^{k,\al}}
\les ( || \na u^{\th} ||_{L^{\infty}} + ||  u^{\th} ||_{L^{\infty}}   ) 
|| \na u^{\th} ||_{X_{O,\s}^{k,\al}}.
\]

Due to \eqref{eq:euler_supp}, $\om^{\th}, u^{\th}, \td \psi \one_{S_{\max}} $ can be seen as  functions in a half plane $ r \leq 1, z \in \R$. Since $r, r^{-1}$ are smooth \eqref{eq:euler_r}
in $S_{max}$, the remaining steps in the proof of Proposition \ref{prop:euler_bd} follows that of 
 Proposition \ref{prop:BKM_bous} and the product rule in Lemma \ref{lem:prod}.

\bibliographystyle{plain}
\bibliography{selfsimilar}

\end{document}

%% file: mymacro.tex
\usepackage{amsmath}
\usepackage{amssymb}
\usepackage{amsthm}
\usepackage{amsmath}
\usepackage{setspace}
\usepackage{xcolor}
\usepackage{fancyhdr}
\usepackage{amssymb}
\usepackage{amsthm}
\usepackage{listings}
    \usepackage{cite}
    \usepackage{tabularx}
    \usepackage{graphicx}
    \usepackage{epstopdf}    
    \usepackage{epsfig}
    \usepackage{float}
    \usepackage{ listings}
    \usepackage{appendix}
 \theoremstyle{plain}
  \newtheorem{thm}{Theorem}
 \newtheorem{prop}{Proposition}[section]
 \newtheorem{lem}[prop]{Lemma}
 
 \theoremstyle{definition}

 \theoremstyle{remark}
 \newtheorem{remark}[prop]{Remark}

 \let\pa=\partial
 \let\al=\alpha
 \let\b=\beta
 \let\d=\delta
 \let\g=\gamma
 \let\e=\varepsilon

 \let\lam=\lambda
 
 \let\s=\sigma
 \let\f=\frac
 \let\inf = \infty
 \let \les = \lesssim
  \let \gtr = \gtrsim
 \let\om=\omega
 
 \let \th = \theta

 \let \vp = \varphi
 \let\G= \Gamma
\let\B = \Big
 \let\D=\Delta

 \let\Om=\Omega
 \let\td = \tilde

 \let\teq \triangleq
 
 \let\pa=\partial
 \let \bsh = \backslash

 \def\cC{{\mathcal C}}

 \def\cH{{\mathcal H}}

 \def\na{\nabla}
 \def\la{\langle}
 \def\ra{\rangle}

\def\one{\mathbf{1}}

 \newcommand{\beq}{\begin{equation}}
 \newcommand{\eeq}{\end{equation}}
  \newcommand{\bal}{\begin{aligned} }
  \newcommand{\eal}{\end{aligned}}
 \newcommand{\ben}{\begin{eqnarray}}
 \newcommand{\een}{\end{eqnarray}}
 \newcommand{\beno}{\begin{eqnarray*}}
 \newcommand{\eeno}{\end{eqnarray*}}

  \newcommand{\BT}{\mathbb{T}}
  \newcommand{\BZ}{\mathbb{Z}}

 \newcommand{\uu}{\mathbf{u}}

 \newcommand{\R}{\mathbb{R}}

 \newcommand{\supp}{\mathrm{supp}}